\documentclass[11pt]{amsart}
\usepackage{geometry}                
\usepackage{graphicx}
\usepackage{amssymb}
\usepackage{epstopdf}
\newtheorem{theorem}{Theorem}[section]
\newtheorem{lemma}{Lemma}[section]
\newtheorem{corollary}{Corollary}[section]

\newtheorem{definition}{Definition}[section]

\numberwithin{equation}{section}

\def\R{\Bbb R}

\def\C{\Bbb C}

\def\d{\partial}
\def\a{\alpha}
\def\b{\beta}
\def\g{\gamma}

\def\e{\epsilon}

\title{Traversally  Generic \& Versal Vector Flows:  Semi-Algebraic Models of Tangency to the Boundary}
\author{Gabriel Katz}

\address{5 Bridle Path Circle, Framingham, MA 01701, USA}
\email{gabkatz@gmail.com}

\begin{document}
\maketitle

\begin{abstract} Let $X$ be a compact smooth manifold with boundary. In this article, we study the spaces $\mathcal V^\dagger(X)$ and $\mathcal V^\ddagger(X)$ of so called boundary generic and traversally generic  vector fields on $X$ and the place they occupy in the space $\mathcal V(X)$ of all fields (see Theorems \ref{th3.4} and Theorem \ref{th3.5}). The definitions of boundary generic and traversally generic vector fields $v$ are inspired by some classical notions from the singularity theory of smooth Bordman maps \cite{Bo}. Like in that theory (cf. \cite{Morin}), we establish local versal algebraic models for the way a sheaf of $v$-trajectories interacts with the boundary $\d X$. For fields from the space $\mathcal V^\ddagger(X)$, the finite list of such models depends only on $\dim(X)$; as a result, it is universal for all equidimensional manifolds. In specially adjusted coordinates, the boundary and the $v$-flow acquire descriptions in terms of universal deformations of real polynomials whose degrees do not exceed $2\cdot \dim(X)$.  
\end{abstract}

\section{Introduction}

This paper is the second in a series that researches the Morse Theory, gradient flows, concavity and complexity on smooth compact manifolds with boundary. In the context of $3D$-flows, some of its ideas can be traced back to \cite{K}. The paper serves as an analytical foundation for the investigation of \emph{boundary generic} (see Definition \ref{def2.1})  and, so called, \emph{traversally generic}\footnote{For vector fields $v$ that vanish on $X$, the $v$-trajectory space is pathological; in contrast,  the trajectory spaces of a traversally generic  fields are a compact $CW$-complexes.} (see Definition \ref{def3.2}) vector fields $v$ on manifolds $X$ with boundary. These analytical tools provide us with local semi-algebraic models for the ways in which typical nonsingular vector flows interact with the boundary $\d X$. Here the word ``local" refers to the vicinity of a given trajectory $\g$ of the $v$-flow. 

The main observation is that, for smooth fields,  each intersection point $a \in \g \cap \d X$ comes with a positive integral multiplicity $j(a)$ attached to it.  This multiplicity $j(a)$ can be given a number of competing but equivalent definitions, one of which uses the \emph{Morse stratification} (see Definition \ref{def2.1} and formula (\ref{eq2.1})), which has been studied in \cite{K1}. Naively, one can think of $j(a)$ as a \emph{multiplicity of tangency} between $\g$ and $\d X$. So, surprisingly, the smooth topology of the flow can distinguish between, say, degree 2 and degree 4 tangency!

Lemma \ref{lem3.1}  and Lemma \ref{lem3.4} describe the models for $\d X$ and $v$ in the vicinity of point $a \in \g \cap \d X$ and in the vicinity of a trajectory $\g$, respectively.  It turns out that, for traversally generic fields, in special flow-adjusted coordinates $(u, \vec x)$, the boundary is given by a real polynomial equation $P(u, \vec x) = 0$ of degree that depends on $\g$ and does not exceed $2\cdot \dim(X)$. The manifold $X$ is given by the polynomial inequality $P(u, \vec x) \leq 0$. The polynomial $P(u, \vec x)$ depends only on the ordered sequence of  multiplicities  $\{j(a)\}_{a \in \g \cap \d X}$. So, in each dimension, there are only finitely many semi-algebraic models for the vicinity of $v$-trajectories $\g$ in $X$.

We introduce a variety of spaces that correspond to different types of vector fields on $X$, the space $\mathcal V^\dagger(X)$ of generic  with respect to the boundary fields and the space  $\mathcal V^\ddagger(X)$ of traversally generic fields are among them. Two theorems describe our main results: Theorem \ref{th3.4} claims that $\mathcal V^\dagger(X)$ is an open and dense in the space of all smooth fields $\mathcal V(X)$, and Theorem \ref{th3.5} claims that $\mathcal V^\ddagger(X)$ is open and dense in the space $\mathcal V_{\mathsf{trav}}(X)$ of all \emph{traversing}  (equivalently, all gradient-like non-vanishing) fields. Traversing fields have only trajectories that are homeomorphic to closed intervals or singletons.

\section{Morin's Local Models: How Nonsingular Flows Interact with Boundary}

Let $v$ be a vector field on a smooth compact $(n+1)$-manifold $X$ with boundary $\d X$. To achieve some uniformity in our notations, let $\d_0X := X$ and $\d_1X := \d X$.

The vector field $v$ gives rise to a partition $\d_1^+X \cup \d_1^-X $ of the boundary $\d_1X$ into  two sets: the locus $\d_1^+X$, where the field is directed inward of $X$, and  $\d_1^-X$, where it is directed outwards. We assume that $v$, viewed as a section of the quotient  line bundle $T(X)/T(\d X)$ over $\d X$, is transversal to its zero section. This assumption implies that both sets $\d^\pm_1 X$ are compact manifolds which share a common boundary $\d_2X := \d(\d_1^+X) = \d(\d_1^-X)$. Evidently, $\d_2X$ is the locus where $v$ is \emph{tangent} to the boundary $\d_1X$.

Morse has noticed that, for a generic vector field $v$, the tangent locus $\d_2X$ inherits a similar structure in connection to $\d_1^+X$, as $\d _1X$ has in connection to $X$ (see \cite{Mo}). That is, $v$ gives rise to a partition $\d_2^+X \cup \d_2^-X $ of  $\d_2X $ into  two sets: the locus $\d_2^+X$, where the field is directed inward of $\d_1^+X$, and  $\d_2^-X$, where it is directed outward of $\d_1^+X$. Again, let us assume that $v$, viewed as a section of the quotient  line bundle $T(\d_1X)/T(\d_2X)$ over $\d_2X$, is transversal to its zero section.

For generic fields, this structure replicates itself: the cuspidal locus $\d_3X$ is defined as the locus where $v$ is tangent to $\d_2X$; $\d_3X$ is divided into two manifolds, $\d_3^+X$ and $\d_3^-X$. In  $\d_3^+X$, the field is directed inward of $\d_2^+X$, in  $\d_3^-X$, outward of $\d_2^+X$. We can repeat this construction until we reach the zero-dimensional stratum $\d_{n+1}X = \d_{n+1}^+X \cup  \d_{n+1}^-X$. 

These considerations motivate 

\begin{definition}\label{def2.1} 
We say that a smooth field $v$ on $X$ is \emph{boundary generic} if:
\begin{itemize}
\item $v|_{\d X} \neq 0$,
\item $v$, viewed as a section of the tangent bundle $T(X)$, is transversal to its zero section,
\item  for each $j = 1, \dots,  n+1$, the $v$-generated stratum $\d_jX$ is a  smooth submanifold of  $\d_{j-1}X$,
\item  the field  $v$, viewed as section of the quotient 1-bundle  $$T_j^\nu := T(\d_{j-1}X)/ T(\d_jX) \to \d_jX,$$ is transversal to the zero section of $T_j^\nu$ for all $j > 0$. 
\end{itemize}
\hfill\qed
\end{definition}

Thus a boundary generic vector field $v$  on $X$  gives rise to two  stratifications: 
\begin{eqnarray}\label{eq2.1}
\d X := \d_1X \supset \d_2X \supset \dots \supset \d_{n +1}X, \nonumber \\ 
X := \d_0^+ X \supset \d_1^+X \supset \d_2^+X \supset \dots \supset \d_{n +1}^+X 
\end{eqnarray}
, the first one by closed submanifolds, the second one---by compact ones.  Here $\dim(\d_jX) = \dim(\d_j^+X) = n +1 - j$. For simplicity, the notations ``$\d_j^\pm X$" do not reflect the dependence of these strata on the vector field $v$. When the field varies, we use a more accurate notation ``$\d_j^\pm X(v)$".

Let $v$ be a boundary generic vector field on  $X$ such that $v \neq 0$ along the boundary $\d X$.  We can add an external collar to $X$ and smoothly extend the field into a larger manifold $\hat X$ without introducing new singularities.  Let $\hat v$ denote the extended field. 

At each point $x \in \d_1X$, the $(-\hat v)$-flow defines the germ of the projection $p_x: \hat X \to S_x$, where $S_x$ is a local section of the $\hat v$-flow which is transversal to it.  The projection is considered at each point of  $\d_1X \subset \hat X$. When $\hat v$ is a gradient-like field for a function $\hat f: \hat X \to \R$, we can choose  the germ of the hypersurface $f^{-1}(f(x))$ for the role of $S_x$.
\smallskip

Let $\mathcal V_{\neq 0}(X)$ be the space of smooth vector fields $v \neq 0$, equipped with the $C^\infty$-topology.  
\smallskip

A theorem of Morin \cite{Morin} describes all local models of $p_x : \d_1X \to S_x$ for a $G_\delta$, or \emph{residual}\footnote{that is, a countable intersection of open and dense subsets in $\mathcal V_{\neq 0}(X)$}, set of fields in $\mathcal V_{\neq 0}(X)$. 

Let us introduce and depict these models.  For any integer $s \in [1, n + 1]$, consider  the polynomial 
\begin{eqnarray}\label{eq2.2}
Q_s(u_1, u_2, \dots ,  u_{n- 1}, \mathsf{u}_n) := \mathsf{u}_n^s + \sum_{i=0}^{s-2}\, u_i \mathsf{u}_n^i
\end{eqnarray}
and the map $\mu_s : \R^n \to \R^{n+1}$ given by 
\begin{eqnarray}\label{eq2.3}
\mu_s : (u_1, u_2, \dots, u_{n- 1}, \mathsf{u}_n) \to 
(u_1, u_2, \dots, u_{n - 1}, Q_s,\, \mathsf{u}_n),
\end{eqnarray}

Let $(y_1, \dots, y_n, y_{n + 1})$ be coordinates in $\R^{n+1}$. The constant field $e_{n+1} := \d_{y_{n + 1}}$,  will play the role of nonsingular field $\hat v$ on $\hat X$.  

Consider  the projection $\pi: \R^{n+1} \to \R^n$ defined by the formula
\begin{eqnarray}\label{eq2.4}
\pi : (y_1, y_2, \dots,  y_n, y_{n+1}) \to (y_1, y_2, \dots, y_n).
\end{eqnarray}
Then the composition $\pi\circ\mu_s$ is given by the formula
\begin{eqnarray}\label{eq2.5}
(u_1, u_2, \dots, u_{n- 1}, \mathsf{u}_n) \to 
(u_1, u_2, \dots, u_{n- 1},\,  \mathsf{u}_n^{s} + \sum_{i=0}^{s-2} u_i\mathsf{u}_n^i).
\end{eqnarray}

Let us denote by $\lambda$ the line distribution $\ker(D\pi)$ tangent to the fibers of the projection $\pi: \R^{n+1} \to \R^n$.
\smallskip

The following result \cite{Morin} is of key importance for our goals. 

\begin{theorem}[\bf Morin]\label{th2.1} For a $G_\delta$-set of 1-dimensional distributions $l$ on $\hat X$ and any point $x \in \d_1 X$, there is a neighborhood $U$ of $x$ in $\hat X$,  a diffeomorphism $h: U \to \R^{n+1}$, and an integer  $s \in [1, n+1]$ such that 
\begin{itemize}
\item $h(x) = 0 \in \R^{n+1}$,
\item $h(\d_1X \cap U) = \mu_s(\R^n)$, where $\mu_s$ is defined by formula (\ref{eq2.3}),
\item the distribution $l|_{U}$ is mapped by the differential $Dh: TU_x \to T\R^{n+ 1}$ to the distribution $\lambda$. \hfill\qed
\end{itemize}
\end{theorem}

Note that  Morin's theorem a priori allows for the four-fold ambiguity: (1) the distribution $l$ can be given two possible orientations (i.e. the model field can be $\pm e_{n + 1}$), and (2) the manifold $X$ can  occupy  each of the two chambers in which $\R^{n+1}$ is locally divided by $ \mu_s(\R^n)$.\smallskip

We can  describe these local models by replacing their $\vec u$-parametric form in formula (\ref{eq2.3}) with equations in  the new coordinates
$$(u, x_0, x_1 \dots, x_{n-1}) := (y_{n + 1}, -y_n, y_1, \dots, y_{n-1}).$$

Let 
\begin{eqnarray}\label{eq2.6}
P_s(u, x) := Q_s(y_1,  \dots , y_{n-1}, y_{n+1}) - y_n = u^s + \sum_{i=0}^{s-2}x_i u^i
\end{eqnarray}
, so that $\d_1X$ is given by the equation 
$P_s(u, x_0, x_1, \dots,  x_{n -1}) = 0$.
For a fixed vector  $x := (x_0, \dots , x_{n -1})$,
$P_s$ is a depressed polynomial in $u$.

Therefore, for each point $x \in \R^n$, the points of the boundary $\d_1X$ residing in the fiber $\pi^{-1}(x)$ of the projection $\pi: \R^{n+1} \to \R^n$  are the real-valued  zeros of  the polynomial $P_s(u, x)$.

With each $x \in \R^n$, we associate the real zero divisor $D_s(x)$ of the $u$-polynomial $P_s(u, x)$. Its support resides in $\R$. We will be particularly interested in the \emph{ordered} sequence $\omega = \{\omega_i\}_i$ of  multiplicities that is generated by the divisor $D_s(x)$. 

Put $e_{n+1} := \frac{\d}{\d u}$. We abbreviate $\big(\frac{\d}{\d u}\big)^j P_s$ as $P_s^{(j)}$.  Let $w := (u, x) \in \R \times \R^n \, \approx\,  \R^{n+1}$.

For the local models as in formula (\ref{eq2.6}), the Morse stratification $\{\d_j^+X := \d_j^+X(\d_u)\}_j$ (see formula (\ref{eq2.1})) can be expressed in terms of the divisor $D_s(x)$ and labeled by the ordered list of multiplicities $\omega(x)$.
\begin{theorem}\label{th2.2} For a $G_\delta$-set of vector fields $v$ on $X$,  $v|_{\d_1X} \neq 0$, and for each point  $a \in \d_1X$, there exists an integer $s \in [1, n+1]$ such that  the models for $v$ and the $v$-induced strata $\{\d_j^+X\}_{1 \leq j \leq n + 1}$, in the vicinity of $a$, are given  by one of the four real semi-algebraic sets:
\begin{enumerate}
\item $X = \{w \in \R^{n+1}|\; P_s(w) \geq 0\}$ and $v = e_{n+1}$;
\begin{itemize}
\item $\d_jX = \{w \in \R^{n+1}|\; P_s^{(i)}(w) = 0 \text{\;for all\;} i < j\}$
\item $\d_j^+X = \{w \in \R^{n+1}|\; P_s^{(i)}(w) = 0\;  \text{\;for all\;} i < j,\;\text{and}\;  P_s^{(j)}(w) \geq  0\}$
\end{itemize} 
\smallskip
\item $X = \{w \in \R^{n+1}|\; P_s(w) \leq 0\}$ and $v = e_{n+1}$;
\begin{itemize}
\item $\d_jX = \{w \in \R^{n+1}|\; P_s^{(i)}(w) = 0 \text{\;for all\;} i < j\}$
\item $\d_j^+X = \{w \in \R^{n+1}|\; P_s^{(i)}(w) = 0\;  \text{\;for all\;} i < j,\;\text{and}\;  P_s^{(j)}(w) \leq  0\}$
\end{itemize}
\item $X = \{w \in \R^{n+1}|\; P_s(w) \geq 0\}$ and $v = -e_{n+1}$;
\begin{itemize}
\item $\d_jX = \{w \in \R^{n+1}|\; P_s^{(i)}(w) = 0 \text{\;for all\;} i < j\}$
\item $\d_j^+X = \{w \in \R^{n+1}|\; P_s^{(i)}(w) = 0\;  \text{\;for all\;} i < j,\;\text{and}\; \break (-1)^j P_s^{(j)}(w) \geq  0\}$
\end{itemize} 
\smallskip
\item $X = \{w \in \R^{n+1}|\; P_s(w) \leq 0\}$ and $v = -e_{n+1}$;
\begin{itemize}
\item $\d_jX = \{w \in \R^{n+1}|\; P_s^{(i)}(w) = 0 \text{\;for all\;} i < j\}$
\item $\d_j^+X = \{w \in \R^{n+1}|\; P_s^{(i)}(w) = 0\;  \text{\;for all\;} i < j,\;\text{and}\; \break (-1)^jP_s^{(j)}(w) \leq  0\}$.
\end{itemize}  
\end{enumerate}
\end{theorem}

\begin{proof}  Our argument is based on Theorem \ref{th2.1}.  By formula (\ref{eq2.6}), the gradient
\begin{eqnarray}\label{eq2.7} 
\nabla P_s = (P_s^{(1)}(u, x),\, 1, u, u^2,\; \dots ,\;  u^{s-2},\, 0, \; \dots \;  0).
\end{eqnarray}
The orthogonality of $e_{n+1}$ and  $\nu_1 := \nabla P_s$ defines the locus $\d_2X$. Hence, the orthogonality is  equivalent to the constraint $P_s^{(1)}(x, u) = 0$, and $\d_2X$ is determined by two equations: $P_s = 0$ and $P_s^{(1)} = 0$. Formulas (\ref{eq2.6}) and (\ref{eq2.7}) imply that the vector $\nu_2 := \nabla [P_s^{(1)}] = \nu_1^{(1)}$ has $P_s^{(2)}$ for its last coordinate. Since $\d_3X$ is the locus where $e_{n+1}$ is tangent to $\d_2X$ and $\nu_2$ is orthogonal the hypersurface $\{P_s^{(1)} = 0\} \supset \d_2X$, the vectors $\nu_2$ and $e_{n+1}$ must be orthogonal along $\d_3X$. This leads to the equation $P_s^{(2)} = 0$. 
 
Using  this type of argument repeatedly for the linear independent vector fields $\{\nu_j := \nu_1^{(j)}\}_j$, proves the first bullets in claims (1) and (2) of the theorem.
\smallskip
 
When $X$ is defined by the inequality $P_s \geq 0$, $e_{n+1}$ points inside $X$ at $w = (u, x) \in  \d_1X$ if and only if the dot product $e_{n+1}\cdot \nu_1(w)  \geq 0$. That is, $\d_1^+X$ is defined by $P_s^{(1)} \geq 0$ together with $P_s = 0$. On the other hand, when $X$ is defined by the inequality $P_s \leq 0$, $e_{n+1}$ points inside $X$ at $w \in  \d_1X$ if and only if $P_s^{(1)}(w) \leq 0$ and $P_s(w) = 0$. 
\smallskip

Note that 

\begin{eqnarray}\label{eq2.8}
\frac{\d}{\d u}(\nu_{j - 1} \cdot e_{n+1}) =  \frac{\d}{\d u}\nu_{j - 1} \cdot e_{n+1} + \nu_{j - 1} \cdot \frac{\d}{\d u}e_{n+1}  = \nu_{j} \cdot e_{n+1} 
\end{eqnarray}
 
Along $\d_2X$, in view of formula (\ref{eq2.8}),  the property $\nu_2 \cdot e_{n+1} \geq 0$ is equivalent to the inequality $\frac{\d}{\d u}(\nu_1 \cdot e_{n+1}) \geq 0$. 
When, along $\d_2X$, the Lie derivative 
$$\mathcal L_{e_{n+1}}(\nu_1 \cdot e_{n+1}) := \frac{\d}{\d u}(\nu_1 \cdot e_{n+1}) =  \nu_2 \cdot e_{n+1}$$ 
is nonnegative, the field $e_{n+1}$ points inside $\d_1^+X$; otherwise, it points inside $\d_1^-X$. Therefore, $\d_2^+X$ is depicted by the inequality $P_s^{(2)} = \nu_2 \cdot e_{n+1} \geq 0$, coupled with the pair of equalities $P_s = 0, P_s^{(1)} = 0$. The general case of $$\d_j^+X = \{P_s^{(i)} = 0\}_{i < j} \cap \{P_s^{(j)} \geq 0\}$$  is analogous.  

The argument in case $(2)$ of the theorem is very similar to case $(1)$. 

We notice that flipping the direction of the field $v = e_{n+1}$ does not change the polarity of the  strata with even $j$'s and reverses the polarity of the strata with odd ones. 

Finally, we combine this description of  stratifications $\big\{\d_j^+\{P_s \geq 0\}\big\}_j$, $\big\{\d_j^+\{P_s \leq 0\}\big\}_j$ with  Morin's Theorem \ref{th2.1} to get the desired local models for a $G_\delta$-set of fields in $\mathcal V(X)$.
\end{proof}

\noindent {\bf Remark 2.1.} Note that Theorem \ref{th2.2} does not describe  local models for generic Morse data $(f, v)$ in the vicinity of a typical point $a \in \d_1X$, just $4(n +1)$ local models for generic nonsingular fields $v$; in other words, $\pm e_{n+1}$ mimics $v$, but the coordinate function $\pm u$ in general does not represent $f: X \to \R$. \hfill\qed
 
\begin{corollary}\label{cor2.1} For the vector fields $v$ that have Morin normal forms as in formula (\ref{eq2.5}) (as in Theorem \ref{th2.1}), all the strata $\d_j X$ are manifolds,  and the embeddings $\d_j X \subset \d_{j-1}X$ are regular. Moreover, these fields  are generic in the sense of Definition \ref{def2.1}.

In the vicinity of each $a \in \d_1X$ of the $s$-type, we get $\d_{s + 1}X = \emptyset$, so that the $s$-type fields are both boundary $(s+1)$-convex and $(s+1)$-concave.
\end{corollary} 

\begin{proof} According to Theorem \ref{th2.2},  $\d_jX$ is given by the system of equations $$\{P_s = 0,\, P_s^{(1)} = 0,\, \dots \, ,\, P_s^{(j-1)} = 0\}.$$ It follows from formulas (\ref{eq2.7}) and (\ref{eq2.8})  that the gradient vector fields $\nu_1, \dots \nu_j$, where $\nu_k = \nabla(P_s^{(k-1)})$,  are linearly independent. Therefore, $\d_jX$ is a submanifold of $\d_1X$. Moreover, since $\nu_j$ is a section of the normal line bundle $T(X_{j -1})/ T(X_j)$ which is transversal to its zero section and since $e_{n+1} = \sum_k (e_{n+1} \cdot \nu_k)\nu_k$, the field $e_{n+1}$ is boundary generic in the sense of Definition \ref{def2.1}.

To validate the last claim, note that $P_s^{(s)} \neq 0$, which implies that $\d_{s+1}X = \emptyset$. 
\end{proof}

\noindent {\bf Remark 2.2.} For local models as in Theorem \ref{th2.2}, the germs of intersections of the strata $\d_jX$, $j \in [1, n+1]$, with the hypersurfaces of $\{u = c\}$ are affine subspaces of $\R^{n+1}$, while the germs of $\d_j^+X \cap \{u = c\}$ are affine half-spaces. 
Indeed, for each fixed value of $u$ and any $j$, the equations $\{P_s^{(i)} = 0\}_{ i < j}$ impose \emph{linear} constraints on the rest of the variables;  similarly, $\pm P_s^{(j)} \geq 0$ is a linear inequality in $x_0, \dots , x_{n-1}$. Therefore, the equations define a \emph{ruled} real variety, and the inequality picks a semi-algebraic subvariety, ruled by the half-spaces. \hfill\qed
\smallskip

\begin{figure}[ht]
\centerline{\includegraphics[height=4in,width=4in]{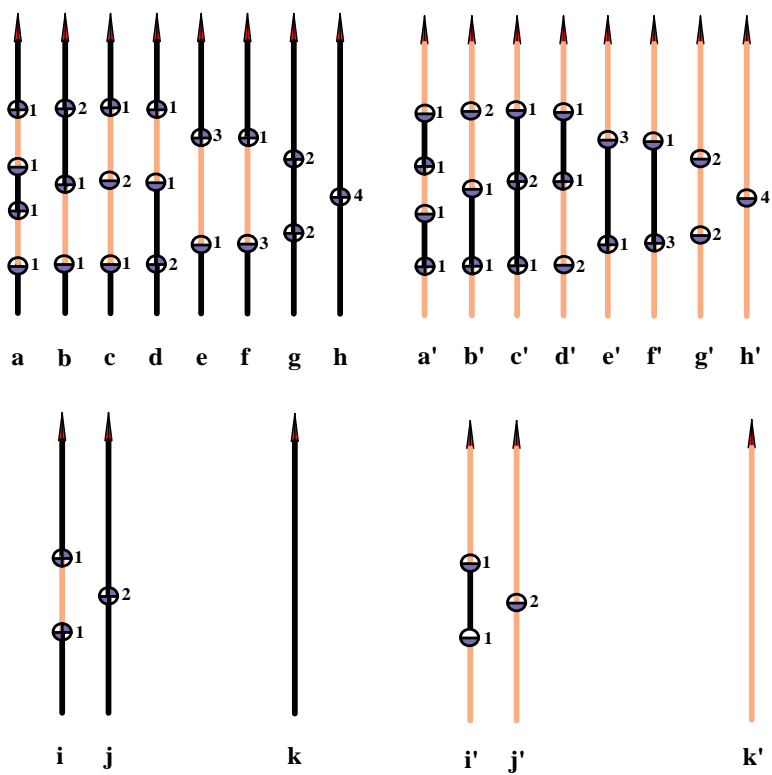}}
\bigskip
\caption{\small{The patterns of solutions for $P_4 \geq 0$ (on the left)  and for $P_4 \leq 0$ (on the right). The numbers 1, 2, 3, 4 indicate the multiplicities of the roots. Diagrams $i, j, i', j'$ correspond to the case of one pair, and diagrams $k, k'$ to the case of  two pairs of complex conjugate roots.}}
\end{figure}
\smallskip

\noindent {\bf Example 2.1} For a 4-dimensional $X$, the eight local models  $$u: \{\pm P_s \geq 0\} \to \R$$, where $s \in [1,4]$ and $v = e_4$\footnote{To save space, we do not list the other eight cases with $v= -e_4$.}, are given by the four polynomials:

\begin{itemize}
\item $P_1 =  u$, 
\item $P_2 =  u^2 + x_0$,
\item $P_3 =  u^3 + x_1 u + x_0$,
\item $P_4 =  u^4 + x_2 u^2 + x_1 u + x_0$.
\end{itemize}

Let us consider the fibers $\pi^{-1}(x) \cap X$ of the projection $\pi: \R^4 \to \R^3$ in the vicinity of the origin for the $P_4$-model. Since the coefficients of $P_4(\sim, u)$ are real, the divisor $D(P_4) \subset \R$  is: (i) either real of degree four, or (ii) real of degree two or (iii) real of degree zero. Note that, in case (i), the sum of all real roots, taken with their multiplicities, is zero. 

The $P_4$-model is described by the diagrams a---k and a'---k' in Fig. 1. In the figure, we do not stress the balanced nature of the divisors $\{P_4 = 0\}$.  The shaded portions of the number lines in the figure belong to the 4-fold $X$; in fact, one can think of $X$ as being disjoint union of these shaded portions. Note the polarities $\{+, -\}$ attached to each root of $P_4$: they reflect the polarities in the Morse stratification $\d_j^\pm X$. By the Vi\`{e}te Formula, each divisor in diagrams a---k and a'---k' determines a unique point $x = (x_0, x_1, x_2)$ over which it resides. \hfill\qed

\section{Traversally Generic and Versal Fields}

Guided by the geometry of  the local models  from Theorem \ref{th2.2}, which describes the ways in which vector fields on $X$ interact with its boundary $\d_1X$, we embark on an investigation of the ``semi-local" dynamics of generic gradient-like \emph{nonsingular} $v$-flows. 
\smallskip

Here and on, we assume that each $v$-trajectory does not ``end" at a point where it is tangent to the boundary; when possible, it ``extends further" in the interior of $X$. Also, for technical reasons, we do regard a singleton $x$, the flow ``curve" through $x \in \d_2^-X$, as a trajectory.\smallskip

For boundary generic (see Definition \ref{def2.1}) vector fields fields $v$---elements of the space $\mathcal V^\dagger(X)$---, we associate an ordered sequence of  multiplicities with each trajectory $\g$, such that $\g \cap \d_1X$ is a finite set.  For any $v \in \mathcal V^\dagger(X)$, the intersection $\{\a_i\} := \g \cap \d_1X$ is automatically a finite set. For traversing (see Definition 4.6 from \cite{K1}) generic fields, the points $\{a_i\}$ of the intersection $\g \cap \d_1X$  are ordered by the field-oriented trajectory $\g$, and the index $i$ reflects this ordering.

\begin{definition}\label{def3.1} Let $v \in \mathcal V^\dagger(X)$ be a generic field. Let $\g$ be a  $v$-trajectory which intersects the boundary $\d_1X$ at a finite number of points $\{a_i\}$. Each point $a_i$ belongs to a unique pure stratum $\d_{j_i}X^\circ$. 

The \emph{multiplicity} $m(\g)$ of $\g$ is defined by the formula
\begin{eqnarray}\label{eq3.1}
m(\g) = \sum_i\, j_i
\end{eqnarray}
The \emph{reduced multiplicity} $m'(\g)$ of $\g$ is defined by the formula
\begin{eqnarray}\label{eq3.2}
m'(\g) = \sum_i (j_i -1)
\end{eqnarray}
, and the \emph{virtual multiplicity} $\mu(\g)$ of $\g$ is defined by 
\begin{eqnarray}\label{eq3.3}
\mu(\g) = \sum_i \Big\lceil\frac{j_i}{2}\Big\rceil
\end{eqnarray}
, where $\lceil\sim\rceil$ denotes the integral part function. \hfill\qed
\end{definition}

For an open and dense subspace $\mathcal V^\ddagger(X)$ of $\mathcal V^\dagger(X)$, one can interpret $\mu(\g)$ as the maximal number of tangency points that any trajectory $\g'$ in the vicinity of $\g$ may have (see Theorem \ref{th3.4}).
\smallskip

When $v$ is nonsingular on $\d_1X$, we can extend it into a larger manifold $\hat X$ so that $\hat X$ properly contains $X$ and the extension $\hat v$ remains nonsingular in in the vicinity of $\d_1X \subset \hat X$. Throughout this text, we treat the pair $(\hat X, \hat v)$ as a germ which extends $(X, v)$.
\smallskip

When $v \in \mathcal V^\dagger(X)$, each set $\d_jX^\circ$ is a manifold.
\smallskip 

Consider the collection of tangent spaces $\{T_{a_i}(\d_{j_i}X^\circ)\}_i$ to the pure strata $\{\d_{j_i}X^\circ\}_i$ that have a non-empty intersection with a given trajectory $\g$. 
 By Theorem \ref{th2.2}, each space $T_{a_i}(\d_{j_i}X^\circ)$ is transversal to the curve $\g$. 
 \smallskip 

Let $S$ be a local section of the $\hat v$-flow at some point $a_\star \in \g$ and let $\mathsf T_\star$ be the  space tangent to $S$ at $a_\star$. Each space $T_{a_i}(\d_jX^\circ)$, with the help of the $\hat v$-flow, determines a vector subspace $\mathsf T_i = \mathsf T_i(\g)$ in  $\mathsf T_\star$. It is the image of  the tangent space $T_{a_i}(\d_jX^\circ)$ under the composition of two maps: (1) the differential of the flow-generated diffeomorphism that maps $a_i$ to $a_\star$ and (2) the linear  projection $T_{a_{\star}}(X) \to \mathsf T_\star$ whose kernel is generated by $v(a_\star)$. 

For a traversing $v$ and a majority  of trajectories, we can choose the space $T_{a_\star}(\d_1^+X)$ for the role of $\mathsf T_\star$, where $a_\star$ is the lowest point of $\g \cap \d_1X$. 
\smallskip

The configuration $\{\mathsf T_i\}$ of \emph{affine} subspaces  $\mathsf T_i \subset \mathsf T_\star$ is called \emph{generic} (or \emph{stable}) when all the multiple intersections of spaces from the configuration have the least possible dimensions consistent with the dimensions of $\{\mathsf T_i\}$. In other words, $$\textup{codim}(\bigcap_{s} \mathsf T_{i_s},  \mathsf T_\star) = \sum_s \textup{codim}(\mathsf T_{i_s},  \mathsf T_\star)$$ for any subcollection $\{\mathsf T_{i_s}\}$ of spaces from the list $\{\mathsf T_i\}$.

Consider the case when $\{\mathsf T_i\}$ are \emph{vector} subspaces of $\mathsf T_\star$.  If we interpret each $\mathsf T_i$ as the kernel of a linear epimorphism $\Phi_i:  \mathsf T_\star \to \R^{n_i}$, then the property of  $\{\mathsf T_i\}$ being generic can be reformulated as the property of  the direct product map $\prod_i \Phi_i:  \mathsf T_\star \to \prod_i  \R^{n_i}$ being an epimorphism.  In particular, for a generic configuration of affine subspaces, if a point  belongs to several  $\mathsf T_i$'s, then the sum of their codimensions $n_i$ does not exceed the dimension of the ambient space $\mathsf T_\star$. 

 \smallskip
\noindent {\bf Example 3.1} The configuration of a line and a plane in $\R^3$, the line being transversal to the plane, is generic; and so is the configuration of two planes in $\R^3$ which share only a line. On the other hand, two lines in $\R^3$ which share a point do not form a generic configuration. Also, any three lines in $\R^2$ which share a point do not form a generic configuration. \hfill\qed
\smallskip

The definition below resembles and is inspired by the ``Condition NC" imposed on, so called, \emph{Boardman maps} between smooth manifolds (see \cite{GG}, page 157, for the relevant definitions). In fact, for  traversing generic fields $v$, the $v$-flow delivers germs of Boardman maps $p(v, \g): \d_1X \to \R^n$, available in the vicinity of every trajectory $\g$. 

\begin{definition}\label{def3.2} We say that a traversing field $v$ on $X$ is \emph{traversally generic} if: 
\begin{itemize}
\item  the field is boundary generic in the sense of Definition \ref{def2.1},
\item for each $v$-trajectory $\g \subset X$ (not a singleton), the collection of  subspaces $\{\mathsf T_i(\g)\}_i$  is generic in $\mathsf T_\star$: that is, the obvious quotient map $\mathsf T_\star \to \prod_i \big(\mathsf T_\star/ \mathsf T_i(\g)\big)$ is surjective.   
\end{itemize}

We denote by $\mathcal V^\ddagger(X)$ the space of all traversally generic fields on $X$. \hfill\qed
\end{definition}

\noindent{\bf Remark 3.1.} In particular, the second bullet of the definition implies the inequality $$\sum_i \textup{codim}(\mathsf T_i(\g), \mathsf T_\star) \leq \dim(\mathsf T_\star) = n.$$   In other words, for traversally generic fields, the reduced multiplicity of each trajectory $\g$ satisfies the inequality
\begin{eqnarray}\label{eq3.4}
m'(\g) = \sum_i (j_i - 1) \leq n.
\end{eqnarray}

Evidently the property of the configuration $\{\mathsf T_i(\g)\}_i$ being generic in $\mathsf T_\star$ does not depend on the choice of the section $S$. \hfill\qed
\smallskip

\noindent{\bf Remark 3.2.} Note that, if a smooth submersion $F: (\tilde X, \d \tilde X) \to (X, \d X)$ is a finite covering of $(n+1)$-manifolds, then the pull-back $\tilde v$ of a traversally generic field $v$ on $X$ is a traversally generic field on $\tilde X$. Indeed, each $v$-trajectory $\g$ is a segment or a singleton. Therefore $F^{-1}(\g)$ is a disjoint union of arcs or points. These are trajectories of the pull-back field $\tilde v$ on $\tilde X$.  Each of the lifted trajectories has a $\tilde v$-adjusted neighborhood $U_{\tilde\g}$ that projects by the local diffeomorphism $F$ onto a $v$-adjusted neighborhood $U_\g$. That field-respecting diffeomorphism  $F: U_{\tilde\g} \to U_\g$  maps $\d\tilde X \cap U_{\tilde\g}$ to $\d X \cap U_\g$, so that all the structures participating in Definition \ref{def3.2} are respected by $F$.

\hfill\qed
\smallskip

Let us review  the list of notations for various kinds of vector fields on $X$. Recall that we have introduced the following nested collection of spaces: 
$$\mathcal V^\ddagger(X) \subset  \mathcal V^\dagger(X) \subset \mathcal V(X)$$
based on traversally generic, boundary generic, and arbitrary smooth vector fields on $X$, respectively. We denote by $\mathcal V_{\neq 0}(X)$ the space of non-vanishing fields on $X$.  We are also considering the space $\mathcal V_{\mathsf{grad}}(X)$ of gradient-like fields. Finally, $\mathcal V_{\mathsf{trav}}(X)$ denotes the space of traversing fields (see Definition 4.6 in \cite{K1}). 

In view of Corollary 4.1 from \cite{K1} the following inclusions follow from the definitions:
\begin{eqnarray}\label{eq3.5}
\mathcal V_{\mathsf{grad}}(X) \cap \mathcal V_{\neq 0}(X) = \mathcal V_{\mathsf{trav}}(X) \subset \mathcal V_{\neq 0}(X), 
\end{eqnarray}

\begin{eqnarray}\label{eq3.6}
\mathcal V^\ddagger(X)  \subset \mathcal V_{\mathsf{trav}}(X) \cap \mathcal V^\dagger(X)
\end{eqnarray}

Recall also that, by the Phillips Theorem B \cite{Ph}, for a fixed Riemannian metric $g$ on $X$, the gradient map $$\nabla_g: Sub(X, \R) \to \mathcal V_{\mathsf{trav}}(X) \subset  \mathcal V_{\neq 0}(X)$$, where $Sub(X, \R)$ denotes the space of submersions $f: X \to \R$, is a \emph{weak homotopy equivalence} between the spaces $Sub(X, \R)$ and $\mathcal V_{\neq 0}(X)$. 

One might speculate that $\nabla_g: Sub(X, \R) \to \mathcal V_{\mathsf{trav}}(X)$ is a weak homotopy equivalence as well. We can prove  (see Corollary 4.1 \cite{K2}) that if two gradient-like fields can be connected by a path in the space of all non-vanishing fields, then they can be connected by a path in the space of all non-vanishing gradient-like fields as well.
Therefore, at least the map $$(\nabla_g)_\ast: \pi_0(Sub(X, \R)) \to \pi_0(\mathcal V_{\mathsf{trav}}(X))$$ is bijective.  
\smallskip

It turns out that, for generic fields $v \in \mathcal V^\dagger(X)$, the way in which each trajectory $\g$ intersects with the Morse strata $\{\d_kX^\circ\}_k$ reflects ``\emph{the order of tangency}" between $\g$ and $\d_1X$. The fundamental lemma below reflects this fact. 

We embed the pair $(X, v)$ into a pair $(\hat X, \hat v)$ in a way that has been described previously.   

\begin{lemma}\label{lem3.1} Assume that $v \in \mathcal V^\dagger(X)$. Denote by $\g_a$ the $\hat v$-trajectory through a point $a \in X$. Let $z: \hat X \to \R$ be a smooth function in the vicinity of $\d_1X \subset \hat X$  such that:  
\begin{enumerate}
\item $0$ is a regular value of $z$,   
\item $z^{-1}(0) = \d_1X$, and 
\item $z^{-1}((-\infty, 0]) = X$. 
\end{enumerate}

Then the following properties hold:
\begin{itemize}
\item For each point $a \in \d_kX^\circ$, the restricted function  $z|_{\g_a}$ has  zero of multiplicity $k$ at the point $a$. 
\item In the vicinity of $a$ in $\hat X$, there exists a coordinate system $(u, x)$, where $u \in\R$ and $x \in \R^n$,  so that: 

1) each $v$-trajectory  $\g$ is defined by an equation $\{x = \vec{const}\}$,

2) the boundary $\d_1X$ is defined by the equation
\begin{eqnarray}\label{eq3.7} 
u^k + \sum_{j =0}^{k -2} \; x_j\, u^j = 0
\end{eqnarray} 
\item In the vicinity of $a \in \d_kX^\circ$, each $v$-trajectory $\g$ hits only some strata $\{\d_jX^\circ\}_{j \in J(a)}$ in such a way that  
$\sum_{j \in J(a)} j \leq k$ and $\sum_{j \in J(a)} j \equiv k\; (2)$. 
\end{itemize} 
\end{lemma} 

\begin{proof}  Consider the space $\R^1 \times \R^n$ with coordinates $(z, x) := (z, x_1, \dots x_n)$ and the standard Euclidean scalar product $\langle \sim , \sim \rangle$. Let the sets $M_0 := \{z \geq 0\}$ and $M_1 := \{z = 0\}$ represent the germ of the pair $(X, \d_1X)$ in the vicinity of a typical point in $\d_1X$. 

Consider a smooth vector field $v(z,  x) = \big(u(z, x),  w(z,  x)\big)$ on  $\R^1 \times \R^n$, where the component $u(z,  x)$ is parallel to $\R^1$, and  $w(z, x)$ to $\R^n$. 

In the argument to follow, we view the coordinates $(z, x)$ as  ``more permanent ingredients", while  analyzing the restrictions on $u(z,  x), w(z, x)$ imposed by the desired property of the field $v(x, z)$ being boundary generic. In fact, we can allow for some changes in the coordinates $(z, x)$ as well, as long as the locus $\{z = 0\}$ (which models the boundary $\d_1X$) remains fixed. 

The field $v$ is tangent to $M_1 \subset \R^1 \times \R^n$ along the locus $$M_2 := \{z = 0,\;  u(z,  x) = 0\}$$ which models $\d_2X := \d_2X(v)$. To reflect the generic nature of $v$, we require that $u(0, x)$, viewed as a section of the obvious projection $\R^1 \times \R^n \to \R^n$, will be transversal to $M_1$ along $M_2$. In other words, the gradient $\nabla_x u$ of the function $u(0, x)$ with respect to the coordinates $x$ should not vanish along the manifold $M_2$. The locus $M_3$, where $v$ is tangent to $M_2$, is given by two equations $\{z = 0\}$ and $\{\langle \nabla_x u ,  w \rangle = 0\}$, while the transversality of $v$ to $M_3$ can be expressed as the condition of linear independency of the two fields, $\nabla_x u$ and $\nabla_x \langle \nabla_x u, w \rangle$. The set  $$M_1^+ := \{x| \; z= 0,\, u \geq 0\} \subset M_1$$ mimics the locus $\d_1^+X$, while the set $$M_2^+ :=  \{x |\; z = 0,\, u = 0,\, \langle \nabla_x u , w\rangle \geq 0\}\subset M_2$$ mimics $\d_2^+X$.

In order to capture the emerging pattern, for each pair of smooth maps $u: \R^1 \times \R^n \to \R$  and $w: \R^1 \times \R^n \to \R^n$, we introduce a sequence of new functions $\psi_k : \R^1 \times \R^n \to \R^1$ in the variables $z, x$ by the recursive formula:
\begin{eqnarray}\label{eq3.8}
\psi_1(z, x) & : = & u(z, x),  \nonumber \\
\psi_k(z, x) & := & \langle \nabla_x \psi_{k-1}(z, x),\, w(z, x) \rangle .
\end{eqnarray}
For $k > 0$, locus $M_k$, a model of $\d_kX$, is given by the equations
\begin{eqnarray}\label{eq3.9}
\{z = 0,\; \psi_1(z, x) = 0,\; \dots \, , \; \psi_{k-1}(z, x) = 0\}
\end{eqnarray}
, together with the requirement that on the solution set $\{(0, x)\}$ of equation (\ref{eq3.9}) the vector fields  
\begin{eqnarray}\label{eq3.10}
\Big\{\nabla_x \psi_1(0, x),\; \dots ,\; \nabla_x \psi_{k-1}(0, x)\Big\}
\end{eqnarray} 
are linearly independent (hence, $M_k$ is a manifold). Thus the linear independence of the the gradient fields in (\ref{eq3.10}), where $2 \leq k \leq n+1$,  is equivalent to $v$ \emph{locally} belonging to the space  $\mathcal V^\dagger(X)$.

The submanifold $M_k^+ \subset M_k$ is defined by the additional inequality $\{\psi_k(z, x) \geq 0\}$: there the field $v$ points inside of $M_{k-1}^+$.
\smallskip 

Consider the following system of ordinary differential equations, determined by the field $v(z, x)$: 
\begin{eqnarray}\label{eq3.11}
\frac{dz}{dt} & = & u(z, x) \nonumber \\
\frac{dx}{dt} & = & w(z, x)
\end{eqnarray}

Take a point $x \in M_k^\circ := M_k \setminus M_{k+1}$. Let $\g(t)= (z(t), x(t))$ be the solution of (\ref{eq3.11}) such that $\g(0) = (x, 0)$. 
In order to prove the lemma, we need to verify  that  the function $z$, being restricted to $\g(t)$, has $x$ as its zero of multiplicity $k$. 

For any $j$, with the help of equations (\ref{eq3.8}) and (\ref{eq3.11}), we get 
 $$\frac{d \psi_j }{dt} = \frac{\d \psi_j}{\d z}\cdot \frac{dz}{dt} +  \Big\langle \nabla_x \psi_j,   \frac{d x}{dt}\Big\rangle  = \frac{\d \psi_j}{\d z}\cdot \psi_1 + \psi_{j+1}$$
Put $\xi_j := \frac{\d \psi_j}{\d z}$. Then the formula above can be rewritten as 
\begin{eqnarray}\label{eq3.12}
\frac{d \psi_j }{dt} = \xi_j \cdot \psi_1 + \psi_{j+1}
\end{eqnarray}

Next, we claim that the function $\frac{d^j z}{(dt)^j}$, being restricted to the integral curve $\g$, can be represented as $\psi_j + \sum_{i =1}^{j-1} \b_i \cdot \psi_i$ for an appropriate choice of smooth functions $\b_i = \b_i(z, x)$. In other words, the functions $\frac{d^j z}{(dt)^j}|_\g$ and $\psi_j|_\g$ are congruent modulo the ideal of $C^\infty(\g, \R)$ generated by $\{\psi_i|_\g\}_{i < j}$. In view of equations  (\ref{eq3.9}),  this  implies that $\frac{d^j z}{(dt)^j}(0) = 0$ for all $j < k$. Since $x \notin  M_{k+1}$ translates as $\psi_k(x) \neq 0$, it  follows that $\frac{d^k z}{(dt)^k}(0) \neq 0$. So  $z|_\g$  indeed has zero of multiplicity $k$ at $\g(0)  \in M_k^\circ$.

We proceed to prove the formula $$\frac{d^j z}{(dt)^j}\big|_\g = (\psi_j + \sum_{i =1}^{j-1} \b_i \cdot \psi_i)\big |_\g$$ by induction in $j$. The inductive step $j \Rightarrow j+1$ is carried out by differentiating with respect to $t$ the identity conjectured for $j$ and then by using formulas (\ref{eq3.12}). The basis of induction is represented by the claim $\frac{dz}{(dt)}|_\g = {\psi_1}|_\g := u|_\g$, the first equation in (\ref{eq3.11}). 

Finally, we notice that all these arguments do not depend on the choice of the auxiliary function $z$, subject to the properties $(1)$-$(3)$ described in the lemma. This proves the claim in the first bullet of the lemma.
\smallskip

Next, we switch to new local coordinates $(u, y)$, amenable to the $v$-flow (in contrast with the previous coordinates $(z, x)$ which were adjusted to the boundary $\d_1X$). Consider a point $a \in M_k^\circ$ and a germ of a $v$-trajectory $\g$ passing through $a$. Let $S$ be a germ of a smooth transversal  flow section at $a$, equipped with coordinates $y := (y_1, \dots , y_n)$, such that $a$ is the origin.  Let $u$ be the coordinate obtained by integrating $v$ so that  $S$ is the locus $\{u = 0\}$.  The procedure for producing the coordinate $u$ (under the name ``$\int_S^x v$") has been described in the proof of Lemma 4.1 from \cite{K1}.

Now we view $z: \hat X \to \R$ as a smooth function of $(u,  y)$. We have shown that $z|_{\g} = z(u, 0)$ has a zero at $u = 0$ of multiplicity $k$. Using the Malgrange  Preparation Theorem \cite{Mal},  there exist an invertible germ of a  smooth function $Q(u, y)$ and a germ of the form $$P(u, y) = u^k + \sum_{j =0}^{k -1} \; \phi_j(y) u^j$$, where $\phi_j(0) = 0$, so that $z = P \cdot Q$. Due to this factorization, in the coordinates $(u, y)$, the locus $\{z = 0\}$---the boundary $\d_1X$---is given by an equation $u^k + \sum_{j =0}^{k -1} \; \phi_j(y) u^j = 0$. We can improve  the latter equation by modifying the coordinates $(u, y)$ (by re-parametrizing the trajectories): just put $\tilde u = u + \frac{1}{k} \phi_{k-1}(y)$ and $\tilde y = y$. Then, in the new coordinates, the equation $z = 0$ is transformed into the equation $$\tilde u^k + \sum_{j =0}^{k - 2} \; \tilde\phi_j(y)\tilde u^j = 0$$, a ``surrogate" of the Morin canonical form (\ref{eq2.2}). 

We still need to bridge the gap between the surrogate and true Morin's canonical forms. With this goal in mind, consider the Jacobi matrix $D\Phi(y)$ formed by the partial derivatives $\big\{\frac{\d}{\d y_i } \tilde\phi_j(y)\big\}_{i, j}$. If $D\Phi(0)$ has the maximal possible rank $k  -1$, then we can choose the functions $\{ \tilde\phi_j(y)\}_{0 \leq j < k- 2}$ for the role of the first new coordinates in the vicinity of the origin. They can be completed to a coordinate system comprising the function $\tilde u$ together with some functions $\{\tilde y_i\}_{k \leq i \leq n}$, drawn from the list of $\{y_i\}_{1 \leq i \leq n}$. 

As before, each trajectory $\g$ is produced by freezing all the coordinates, but  $\tilde u$.  In the new coordinates $(\tilde\phi_0,  \dots , \tilde\phi_{k-2}, \tilde y_{k}, \dots , \tilde y_{n}, \tilde u)$,  the germ of $\d_1X$ at $a$ is given by the depressed polynomial equation $\tilde u^k + \sum_{j = 0}^{k - 2} \; \tilde\phi_j\tilde u^j = 0$, the true Morin  canonical form. Therefore $\textup{rk} (D\Phi(0)) = k - 1$ implies the existence of the Morin canonical coordinates in the vicinity of $a$. 

In fact (see Lemma \ref{lem3.3}), the requirement $\textup{rk} (D\Phi(0)) = s$ has an intrinsic meaning. In particular, it is independent on the factorization $z = P \cdot Q$. 
\smallskip

Now consider a trajectory $\g := \{y = c\}$, where $c := (c_1, \dots , c_n)$ is a fixed point, in the vicinity of $a \in \d_1X \cap \g$ and the zeros of the function $z$, being restricted to $\g$. We claim that the two loci 
$$\big\{z(u, c) = 0,\, \frac{\d}{\d u}z(u, c) = 0,\; \dots \, ,\;  \frac{\d^j}{\d u^j}z(u, c) = 0\big\},$$  
\begin{eqnarray}\label{eq3.13}
\big\{P(u, c) = 0,\, \frac{\d}{\d u}P(u, c) = 0,\; \dots ,\;  \frac{\d^j}{\d u^j}P(u, c) = 0\big\} 
\end{eqnarray}
coincide. The argument,  inductive in $j$,  employes the fact that $Q(u, c) \neq 0$.

Since we have established the validity of the first bullet in Lemma \ref{lem3.1}, we get that both systems of equations in (\ref{eq3.13}) locally define the stratum $\d_{j +1}X$.  Moreover, as in the proof of Theorem \ref{th2.2} (in particular, see formula (\ref{eq2.7})), the fact that $v \in \mathcal V^\dagger(X)$, in the vicinity of $a \in \d_kX^\circ$, can be translated as the property of linear independence for the gradients
\begin{eqnarray}\label{eq3.14}   
\nabla z,\; \nabla\frac{\d z}{\d u},\; \dots ,\; \nabla\frac{\d^{k - 2}z}{\d u^{k - 2}} 
\end{eqnarray}
in the coordinates $(u, z)$ at the origin $(0,0)$\footnote{In the coordinates $(u, y)$, this is an analogue of (\ref{eq3.10}).}. Note that these are the gradients of  functions in (\ref{eq3.13}) that determine the stratum $\d_kX$.

We claim that this linear independence is equivalent to the linear independence of the gradients
\begin{eqnarray}\label{eq3.15} 
\nabla P,\, \nabla\frac{\d P}{\d u},\, \dots ,\, \nabla\frac{\d^{k - 2}P}{\d u^{k - 2}}.
\end{eqnarray}

In order to validate this claim, we introduce the $u$-parameter curve $\a(u):= (1, u, \dots , u^{k-2})$. Note that the gradients in (\ref{eq3.15})  admit the following representations:
\[
\begin{array}{c}
\nabla P = \big(\frac{\d  P}{\d u},\, D\Phi \cdot \a\big),  \; \;
\nabla \frac{\d P}{\d u} = \big(\frac{\d^2  P}{\d u^2},\, D\Phi \cdot \frac{\d\a}{\d u}\big),\;\;  \dots \;\; , \\ 
\\
\nabla \frac{\d^{k-2} P}{\d u^{k-2}} = \big(\frac{\d^{k-1}P}{\d u^{k-1}},\, D\Phi \cdot \frac{\d^{k-2}\a}{\d u^{k-2}}\big)
\end{array}
\]
, where $\big(\frac{\d^{j}P}{\d u^{j}},\, D\Phi \cdot \frac{\d^{j-1}\a}{\d u^{j-1}}\big)$ is the vector whose first component is $\frac{\d^{j}P}{\d u^{j}}$, and $D\Phi \cdot \frac{\d^{j-1}\a}{\d u^{j-1}}$ denotes the multiplication of the matrix $D\Phi$ by the vector $\frac{\d^{j-1}\a}{\d u^{j-1}}$. 

Since $\frac{\d^j  P}{\d u^j}(0, 0) = 0$ for all $1 \leq j \leq k-1$, we conclude that  the dimension of the space spanned by vectors  $$\nabla P(0,0), \nabla\frac{\d P}{\d u}(0,0), \dots , \nabla\frac{\d^{k - 2}P}{\d u^{k - 2}}(0,0)$$  equals  the dimension of the space spanned by vectors 
 $$D\Phi \cdot \a,\;  D\Phi \cdot \frac{\d\a}{\d u},\; \dots \; ,\;  D\Phi \cdot \frac{\d^{k-2}\a}{\d u^{k-2}},$$ 
being evaluated at $(0, 0)$. Since $\a$ and its $u$-derivatives are independent vectors, the dimension of the latter space is equal to the rank of $D\Phi(0)$. Now we employ Lemma \ref{lem3.3} below to conclude that the dimension of the space spanned by vectors in formula (\ref{eq3.14}) is the same as the dimension of the space spanned by vectors in  (\ref{eq3.15}), both dimensions being equal to $\textup{rk}(D\Phi(0))$. 
 
If $v \in \mathcal V^\dagger(X)$,  the vectors in formula (\ref{eq3.15}) must be independent at the origin $0$. Therefore,  $\textup{rk}(D\Phi(0)) = k -1$, which implies the existence of the Morin coordinates in the vicinity of $a\in \d_kX^\circ$.  This proves the claim in the second bullet of the lemma.
\smallskip

In view of formula (\ref{eq3.13}),  there exists $\e > 0$, so that the sum of  zero multiplicities of the smooth $u$-function $z(u, c)$ in the interval   $-\e < u < \e$,  is the same as the sum of  zero multiplicities  of the function $P(u, c)$ in the same interval; the later function is a polynomial  of degree $k$ in $u$. Therefore, in the vicinity of $a$, the sum of zero multiplicities of $z|_{\g}$ does not exceed $k$. In other words, if  $a \in M_k^\circ$, then each trajectory $\g$, in the vicinity of $a$, hits only some strata $\{M_j^\circ\}$ so that $\sum j \leq k$. 

Consider the polynomial family $P(u, x)$ in the formula (\ref{eq3.7}). There is $\delta > 0$ such that $\|c\| < \delta$ implies that  $|x_j(c)| < (\e/\rho)^j$, where the universal constant $\rho$ is introduced in  the proof of Lemma \ref{lem3.2} below. By that lemma,  for such $\delta > 0$  all the real roots $u$ of $P(u, x)$ are confined to the interval $(-\e, \e)$. Therefore, each trajectory $\g = \{x = c\}$, passing through the section $S := \{u = 0,\, \|c\| < \delta \}$, has \emph{all} the roots of $P(u, c)$ concentrated in the interval $ (-\e, \e)$. 

Let $D_\g$ be the zero divisor of $P(u, c)$.  Thus, not only $\deg(D_\g) \leq k$, but also $\deg(D_\g) \equiv k \; (2)$. The third bullet of Lemma \ref{lem3.1} has been validated.
\end{proof}

\begin{lemma}\label{lem3.2} Put $x := (x_0, \dots , x_{k-1})$. Let $P(u, x) = u^k + \sum_{j =0}^{k-1} x_j u^j$ be a $x$-parameter family of real monic $u$-polynomials. Then, there exists a universal constant $\rho := \rho(k) > 0$ with the following property:  for any $\e > 0$ and each polynomial $P(u, x)$ with the coefficients $\{|x_j| < (\e/\rho)^j\}$, all the real roots $u$ of $P(u, x)$ are located in the interval $(-\e, \e)$.
\end{lemma}

\begin{proof} Consider the complex Vi\`{e}te map $V: \C^k \to \C^k$ defined by elementary symmetric polynomials $\sigma_1, \dots , \sigma_k$ in the variables-roots $\a_1, \dots \a_k$. 

For each $\b > 0$, consider the set $K(\b) :=\{|\sigma_j| < \b^j\}_{1 \leq j \leq k}$ in the complex space $\C^k$ with the coordinates $\a_1, \dots \a_k$.  Because each polynomial $\sigma_j$ is homogeneous of degree $j$, every real ray from the origin intersects with $K(\b) \subset \C^k$ along a segment. Indeed, any ray that does not belong the coordinate hyperplanes, hits the hypersurface $\{\sigma_k = \b^k\}$ at a singleton, any ray which belongings to the coordinate hyperplanes hits the hypersurfaces $\{\sigma_{k -1} = \b^{k - 1}\}$ at a singleton, and so on ...  Note that $K(\b')$ can be obtained from $K(\b)$ by a conformal scaling with the scaling factor $\b'/\b$. 
 
Denote by $B$ the unit polydisk $\{|\a_j| \leq 1\}_j$ in the complex space $\C^k$ with the coordinates $\a_1, \dots \a_k$.  For each point $\a \in \d B$, consider the real ray $r(\a)$ through $\a$ emanating from the origin and the proportion $\rho(\a)$ between the lengths of segments $[0,\, r(\a) \cap \d K(1)]$ and $[0,\, r(\a) \cap \d B]$. 

Put $\rho  := \max_{\a \in \d B} \; \rho(\a)$. Now, if $\a \in \d K(\b)$, then each $|\a_j| \leq \b/\rho$. 

For a given $\e > 0$, take  the coefficients $x_j = \sigma_j$ so that  $|x_j| < (\e/\rho)^j$.  For such choice of coefficients, all the roots $\{\a_j\}$ of the polynomial $$P(u, x) := u^k + \sum_{j =0}^{k-1} x_j u^j$$ are confined to the interval $(-\e, \e)$.
\end{proof}

Lemma \ref{lem3.1}, as well as well as some future arguments, relies on the technical lemma below. 

\begin{lemma}\label{lem3.3} Let $\{\a_i\}$ be a finite set of distinct real numbers. Let $z(u, y)$ be a smooth function, where $u \in \R$ and $y = (y_1, \dots y_n) \in \R^n$. Assume that $z(u, y)$ is a product of two smooth functions, $$P(u, y) = \prod_i \; \Big[(u -\a_i)^{k_i} + \sum_{l =0}^{k_i-2} \phi_{il}(y) (u -\a_i)^l\Big]$$ and $Q(u, y)$, where  $Q(\a_i, 0) \neq 0$ and  $\phi_{il}(0) = 0$. Put $m := \sum_i (k_i -1) \leq n$.  

For each $i$, consider the $(k_i -1) \times n$-matrix $$D\Phi_i := \Big(\frac{\d \phi_{il}}{\d y_m}(0)\Big).$$ Let $D\Phi$ be the $m \times n$-matrix obtained by stocking all the matrices $\{D\Phi_i\}_i$ one on top of the other. 

Consider the matrix $$M_i(z) := \Big(\frac{\d^{l+1}z}{\d u^l \d y_m}(\a_i,0)\Big)$$, where $0 \leq l \leq k_i- 2$ and $1 \leq m \leq n$. Denote by $M(z)$ the matrix formed by stocking all the matrices $\{M_i(z)\}_i$ one on top of the other.
\smallskip

Then the rank of the matrix $D\Phi$ equals to the rank of the matrix $M(z)$. 
As a result, the restriction $\textup{rk}(D\Phi) \leq s$, where $s \leq n$, is given by algebraic constraints imposed on the coefficients of the degree $k_i$ Taylor polynomials  of the function $z(u, y)$ at $(\a_i, 0)$. In other words, the constraint  $\textup{rk}(D\Phi) \leq s$ defines a closed affine subvariety in the space $$\prod_i \textup{Jet}^{k_i}\big((\R \times \R^n, \a_i \times 0); (\R, 0)\big)$$, the product  of $k_i$-jet spaces. 
\end{lemma}

\begin{proof} To save the pain of complex multi-indexing,  we will prove one special case of the lemma, when the polynomial $P(u, y)$ consists of a single multiplier $$(u -\a_1)^{k_1} + \sum_{l =0}^{k_1-2} \phi_{1l}(y) (u -\a_1)^l.$$ 
  
By shifting all the relevant functions by $\a_1$, we can replace the variable $u - \a_1$ with the variable $u$. Also put $k := k_1$.

The argument is an induction $k \Rightarrow k+1$ by the degree $k$ of the polynomial $$P(u, y) = u^k + \sum_{l =0}^{k-2} \phi_{k,l}(y) u^l$$,  with $k = 2$ being the base of induction\footnote{Note that here the coefficient $\phi_{k,l}$ is not the coefficient $\phi_{kl}$, present in the formulation of Lemma \ref{lem3.3}!}.  

Consider the polynomial $$P_{k +1}(u, y) := u^{k +1} + \sum_{l =0}^{k-1} \phi_{k+1, l}(y) u^l$$ with some smooth functional coefficients, where $\phi_{k+1, l}(0) = 0$.  

Let $z_{k + 1} := P_{k+1}\cdot Q_{k+1}$ with $Q_{k+1}(0, 0) \neq 0$.
Then $$z_{k + 1} = \Big[u\big(u^k + \sum_{l = 1}^{k - 1} \phi_{k+1, l}(y) u^{l-1}\big) + \phi_{k+1, 0}(y)\Big] Q_{k+1}(u,y)$$ 
Define  $$P_k := u^k + \sum_{l = 1}^{k - 1} \phi_{k+1, l}(y) u^{l-1}, \;\; \phi := \phi_{k+1, 0}, \;\; Q := Q_{k+1},$$ and $z_k := P_k \cdot  Q$. 

In the new notations, $$z_{k+1} = u\cdot z_k + \phi \cdot Q.$$

Then $$\frac{\d z_{k+1}}{\d y}(u, 0) = u \cdot \frac{\d z_k}{\d y}(u, 0)+ \frac{\d \phi}{\d y}(0) \cdot Q(u,0), $$
where $\frac{\d}{\d y} := (\frac{\d}{\d y_1}, \dots , \frac{\d}{\d y_n})$. Applying the operators $\frac{\d^j}{\d u^j}$ to the previous identity, for $j \leq k$, we get 
$$\frac{\d^{j+1}z_{k+1}}{\d u^j \d y}(u,0) = j \frac{\d^jz_k}{\d u^{j-1} \d y}(u,0) + u \frac{\d^{j+1}z_k}{\d u^j \d y}(u,0) + \frac{\d\phi}{\d y}(0)  \cdot \frac{\d^j Q}{(\d u)^j}(0,0).$$
The substitution $u = 0$ leads to the equations 
\begin{eqnarray}\label{eq3.16}
\frac{\d z_{k+1}}{\d y}(0, 0) & = & \frac{\d\phi}{\d y}(0)  \cdot Q(0,0),\nonumber  \\ 
\frac{\d^{j+1} z_{k+1}}{\d u^j\d y}(0, 0) & = & j \frac{\d^j z_k}{\d u^{j-1}\d y}(0, 0) +  \frac{\d\phi}{\d y}(0)  \cdot \frac{\d^j Q}{\d u^j}(0,0)
\end{eqnarray}
By inductive assumption, the rank of the $(k-1)\times n$-matrix $M(z_k)$ whose rows are the vectors 
$$\big\{\frac{\d z_k}{\d y}(0,0),\, \frac{\d^2 z_k}{\d u\d y}(0,0),\, \dots\,  ,\, \frac{\d^{k-2} z_k}{\d u^{k-2}\d y}(0, 0)\big\}$$ is equal to the rank of the Jacobi matrix $N(P_k)$ whose rows are the vectors $$\big\{ \frac{\d \phi_{k+1,1}}{\d y}(0),\, \frac{\d \phi_{k+1, 2}}{\d y}(0),\, \dots\, ,\, \frac{\d \phi_{k+1, k-1}}{\d y}(0)\big\}.$$ 

Let $\hat M(z_k)$ be the $(k\times n)$-matrix obtained from $M(z_k)$ by adding the first row of zeros. Consider the column-vector $$q := \big\{Q(0,0),\, \frac{\d Q}{\d u}(0,0),\, \dots , \,\frac{\d^{k-1} Q}{\d u^{k -1}}(0,0)\big\}.$$ Then $$M(z_{k+1}) = \hat M(z_k) + q \times  \frac{\d\phi}{\d y}(0)$$
, where ``$\times$" stands for the matrix multiplication.

Remembering that $Q(0,0) \neq 0$ and that $\phi := \phi_{k+1, 0}$, in view of the equations above (or equations (\ref{eq3.16})), we get that the space spanned by the vectors $$\frac{\d z_{k+1}}{\d y}(0, 0)\; \text{and}  \; \Big\{\frac{\d^{j+1} z_{k+1}}{\d u^j\d y}(0, 0)\Big\}_{1 \leq j \leq k-1}$$ coincides with the space spanned by the vectors $$\frac{\d\phi}{\d y}(0), \frac{\d z_k}{\d y}(0, 0),\; \text{and}  \;\Big\{\frac{\d^{j+1} z_k}{\d u^j\d y}(0, 0)\Big\}_{1 \leq j \leq k-2}.$$ Therefore, $\textup{rk}(M(z_{k+1})) = \textup{rk}(N(P_{k+1}))$. 

We let the reader to verify the validity of the lemma for $k = 2$.
\end{proof}

\begin{figure}[ht]
\centerline{\includegraphics[height=4in,width=4in]{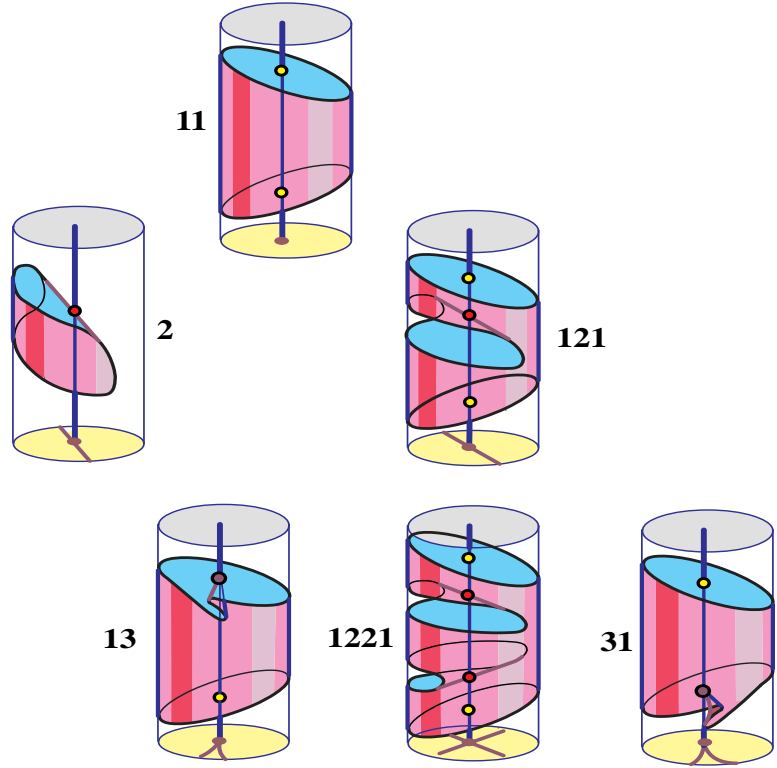}}
\bigskip
\caption{\small{The six geometrical patters of traversally generic fields on $3$-folds}}
\end{figure}

\begin{definition}\label{def3.3} Let $\hat v$ be a nonvanishing vector field in $\hat X$. We say that a set $U \subset \hat X$ is $\hat v$-\emph{adjusted}, if its intersection with every  $\hat v$-trajectory $\hat \g$ is an interval.  \hfill\qed 
\end{definition}

\begin{definition}\label{def3.4} Let $v$ be a traversing and generic vector field in $X$, and $\g$ its trajectory. Consider the intersection $\g \cap \d_1X = \coprod_{1 \leq i \leq p} a_i$, where each point $a_i \in \d_{j_i}X(v)^\circ$. Pick a smooth transversal section $S$ of the $\hat v$-flow at a point $b \in \g$.  

We say that $v$ is \emph{traversally generic at}  $\g$, if  the images of the tangent spaces $\{T_{a_i}(\d_{j_i}X(v)) \}_i$ under the $\hat v$-flow in $T_b(S)$ are in general position. \hfill\qed 
\end{definition}

Fig. 2 shows all the local patterns of traversally generic fields for $3$-folds $X$; the sequences of integers next to the figures label the corresponding tangency patterns of the core trajectory.
\smallskip

The following proposition is a ``semi-global"  generalization of  Morin's Theorem \ref{th3.1}.

\begin{lemma}\label{lem3.4} Let  $\hat v$ be a traversing generic field in a $\hat v$-adjusted neighborhood $U \subset \hat X$ of a given $v$-trajectory $\g \subset X$. Assume that $v$ is traversally generic \emph{at} $\g$. 

Then $\g$ has a $\hat v$-adjusted neighborhood $V \subset U$ with a special system of coordinates 
$$(u,\, \underbrace{x_{10}, \dots  , x_{1j_1-2}},\, \dots \, ,\underbrace{x_{i0}, \dots , x_{ij_i - 2}}, \, \dots \, ,\underbrace{x_{p0}, \dots \, x_{pj_p - 2}},\, \underbrace{y_1, \dots, y_{n - m'(\g)}})$$
such that:
\begin{itemize}
\item $\{u = const\}$ defines a transversal section of the $\hat v$-flow,
\item each $\hat v$-trajectory in $V$ is produced by fixing all the coordinates $\{x_{il}\}$ and $\{y_k\}$,
\item there is $\e > 0$ such that $V \cap \d_1X \subset \coprod_i V_i$, where  $$V_i :=  u^{-1}((\a_i -\e, \a_i + \e)) \cap V$$,  and $\a_i = u(a_i)$,
\item for each $i$, the intersection  $V_i \cap \d_1X$ is given by the equation
\begin{eqnarray}\label{eq3.17}
(u - \a_i)^{j_i} + \sum_{l=0}^{j_i - 2} x_{i, l} (u - \a_i)^l  = 0.
\end{eqnarray}
\end{itemize}
\end{lemma}

\begin{proof} Let $U \subset \hat X$ be a $\hat v$-adjusted neighborhood of a typical trajectory $\g \subset X$.  With the help of the flow, we introduce  a system of coordinates $(u, y) := (u, y_1, \dots y_n)$ in $U$ so that:  
\begin{itemize}
\item $S_i= \{u = \a_i\}$ is a transversal section of $\hat v$-flow, which contains $a_i$ and is diffeomorphic to a closed disk $D^n$, 
\item the locus $\{y = 0\} \cap X$ is the trajectory $\g$, and 
\item each $\hat v$-trajectory in $U$ is given by freezing some point $y$. 
\end{itemize}
Consider a smooth function $z: \hat X: \to \R$, or rather its germ in the vicinity of $X$,  such that:
\begin{itemize}
\item $0$ is a regular value for $z$,
\item $z^{-1}(0) = \d_1X$,
\item $z^{-1}((-\infty, 0]) = X$.
\end{itemize}
\begin{eqnarray}\label{eq3.18}
\end{eqnarray}

In the vicinity of each point $a_i \in \g \cap \d_1X$, the coordinates $(u - \a_i, y) := ( u - \a_i, y_1, \dots y_n)$ are available. As in the proof of Lemma \ref{lem3.1},  in some $\hat v$-adjusted neighborhood $W_i \subset U$ of each point $a_i$, the globally-defined  function $z$ as in (\ref{eq3.18}) can be written as the product of two smooth functions $P_i(u -\a_i, y)\cdot Q_i(u -\a_i, y)$, where $Q_i(u -\a_i, y) \neq 0$, and the polynomial $$P_i(u -\a_i, y) = (u -\a_i)^{j_i} + \sum_{l = 0}^{j_i - 2} \phi_{i,l}(y) (u - \a_i)^l.$$ Its smooth functional coefficients $\{\phi_{i,l}(y)\}_l$ are such that $\phi_{i,l}(0) = 0$. 

We can adjust the size of $W_i$'s so that, for each pair $W_i$ and $W_k$, the set of trajectories passing through  $W_i$ and the  set of trajectories passing through  $W_k$ coincide. In other words, we can find a $\hat v$-adjusted cylindrical neighborhood $W \subset U$ of $\g$, such that $W \cap W_i = W_i$ for all $i$.

Let $\Phi_i: \R^n \to \R^{j_i - 1}$ be given by the functions $\{\phi_{i,l}(y)\}_l$. So we can assume that, in such $\hat v$-adjusted neighborhood $W$  of $\g$, the locus $\{z = 0\}$ is given by one of the polynomial equations $\{P_i(u -\a_i, y) = 0\}_i$.

As in the proof of Lemma \ref{lem3.1} and by Lemma \ref{lem3.3},  $v \in \mathcal V^\dagger(X)$ implies that $\textup{rk}(D\Phi_i(0)) = j_i - 1$.  

Since the $\hat v$-flow is assumed to be traversally generic at $\g$, the flow-generated images $\{\mathsf T_i\}_i$ of all tangent spaces  $\{T_{a_i}(\d_{j_i}X(v))\}_i$ of the minimal strata $\{\d_{j_i}X(v) \ni a_i\}_i$ must be in general position in the tangent space $\mathsf T$ of the section $S$ at the point  $a = \g \cap S$ ($a$ resides below the lowest point $a_1 \in \g \cap \d_1X$).  The space $\mathsf T_i$  is of  the codimension $j_i -1$ in $\mathsf T$.  
 
Let $\Phi:  \R^n \to \R^{m'(\g)}$ be the direct product of the maps $\Phi_i$, where $m'(\g) := \sum_i\, (j_i - 1)$. Then, by the definition of true genericity at $\g$,  $\textup{rk}(D\Phi(0)) = m'(\g) \leq n$ (see Remark 3.1). 
\smallskip

We claim that if $\textup{rk}(D\Phi(0)) = m'(\g)$, then the desired system of coordinates in a new $\hat v$-adjusted neighborhood $V \subset W$ of $\g$ is available: indeed, just put $x_{il} = \phi_{i,l}(y)$ and keep certain $(n - m'(\g))$-tuple of the original coordinates $\{y_k\}$ unchanged.  These coordinates $\{y_k\}$ are chosen so that, together with $\Phi$, they define a map $\R^n \to \R^{m'(\g)} \times \R^{n- m'(\g)}$ of the maximal rank $n$ at the  origin. 
\end{proof} 

\noindent{\bf Remark 3.3.} Lemma \ref{lem3.4} implies that, in such special coordinates,  the intersection $V \cap \d_1X$ is given  by the single polynomial equation
\begin{eqnarray}\label{eq3.19}
P(u, x):= \prod_i \;\big[(u - \a_i)^{j_i} + \sum_{l=0}^{j_i - 2} x_{i, l} (u - \a_i)^l\big]  = 0
\end{eqnarray}
of degree $m(\g) = \sum_i\, j_i$, and $V \cap X$ --- by the inequality $P(u, x) \leq 0$. \hfill\qed
\smallskip

\begin{lemma}\label{lem3.5} Let the smooth manifold $X \subset \R \times \R^d$ be given by  the polynomial inequality $\{P(u, x) := u^d + \sum_{k =0}^{d-1} x_k u^k \leq 0\}$. 

Then the field $\d_u$ is traversally generic with respect to the boundary $\d X = \{P(u, x) = 0\}$.
\end{lemma}

\begin{proof} Let $\pi: \R \times \R^d \to \R^d$ be the obvious projection.

For every point $w := (u, x) \in \R\times\R^d$,  there exists a unique integer $j(w) \geq 0$ such that $$P^{(k)}(w) = 0\;\; \text{for all}\;\;  0 \leq k < j(w), \;\; \text{and}\;\; P^{(j(w))}(w) \neq 0$$, where $P^{(j)}$ denotes the $j$-th partial derivative of $P$ with respect to $u$. Evidently, $j(w) \leq d$. 

We denote by $\Sigma_{[j]}\subset \R\times\R^d$  the locus $\{w |\; j(w) \geq j\}$. It is given by the $j$ equations 
\begin{eqnarray}\label{eq3.20} 
\{P(u, x)=0,\, P^{(1)}(u, x) = 0,\, \dots ,\, P^{(j-1)}(u, x) = 0\}
\end{eqnarray}
Consider any $t$-parameter curve  $\g(t) :=  (u(t), x(t)) \subset \Sigma_{[j]}$ and let us compute the restrictions imposed on its tangent vector $\dot\g(0)$ at the point $\g(0)$.  Differentiating the equations (\ref{eq3.20}) that define $\Sigma_{[j]}$ with respect to $t$, we get :
\begin{eqnarray}\label{eq3.21} 
P^{(1)} \dot u  +  P_x\cdot \dot x = 0, \nonumber \\
P^{(2)} \dot u  +  P_x^{(1)}\cdot \dot x  = 0, \nonumber \\
\dots \nonumber \\
P^{(j)} \dot u  +  P_x^{(j-1)}\cdot \dot x = 0
\end{eqnarray}
, where $P_x$ denotes the row vector $(u^{d-1}, \dots , u, 1)$ and $\dot x$  the column vector $(\dot x_{d-1}, \dots , \dot x_1, \dot x_0)$. By combining (\ref{eq3.21})  with the equations (\ref{eq3.20}) of  $\Sigma_{[j]}$, we get 
\begin{eqnarray}\label{eq3.22} 
P_x\cdot \dot x = 0, \nonumber \\
P_x^{(1)}\cdot \dot x = 0, \nonumber \\
 \dots \nonumber \\
P_x^{(j-2)}\cdot \dot x = 0,  \nonumber \\
P^{(j)} \dot u  + P_x^{(j-1)}\cdot \dot x = 0
\end{eqnarray}
These equations imply that for any $\dot x$, subject to the first $j - 1$ equations in (\ref{eq3.22}), there exists a unique value  $\dot u =  -(P_x^{(j-1)}\cdot \dot x)/P^{(j)}$, unless $P^{(j)}(u, x) = 0$. The latter vanishing does not happen in the proper stratum $\Sigma_{[j]}^\circ := \Sigma_{[j]} \setminus \Sigma_{[j+1]}$.

Therefore, for each vector $(\dot u, \dot x) \in T_{(u,x)}\Sigma_{[j]}^\circ$, tangent to $\Sigma_{[j]}^\circ$ at a point $(u, x)$, the projection $D\pi((\dot u, \dot x)) = \dot x$ belongs to an  vector subspace $V_{[j]}(u, x)\subset T_x(\R^d) \approx \R^d$ defined by the first $j - 1$ equations in (\ref{eq3.22}). Moreover, every vector $\dot x$ at $x$ that belongs to $V_{[j]}(u, x)$ has a unique preimage in $T_{(u,x)}\Sigma_{[j]}^\circ$,  so that $D\pi: T_{(u,x)}\Sigma_{[j]}^\circ \to V_{[j]}(u, x)$ is a linear isomorphism.
\smallskip

Let us pick now a generic point $x \in \R^d$. Let  $L_{x}$ denote the line $\pi^{-1}(x) \subset \R\times\R^n$. Put $Z_{x} := L_{x} \cap \{P = 0\}$.  The $u$-coordinate maps $Z_{x}$ to the set of zeros $\{\a_i = \a_i(x)\}_i$ of the $u$-polynomial $P(u, x)$. We denote by $j_i  = j(\a_i, x)$ the multiplicity of the root $\a_i$. 
\smallskip

Our goal is to show that $\pi$ maps the tangent spaces $\{T_{(\a_i, x)}\Sigma_{[j_i]}\}_i$  to a general position configuration in $\R^d$. We have shown that the $D\pi$  maps bijectively each tangent space $T_{(\a_i, x)}\Sigma_{[j_i]}^\circ$ to the vector space $V(\a_i, x) := V_{[j_i]}(\a_i, x)$---the solution set of the homogeneous linear system $L(\a_i, j_i)$:
\begin{eqnarray}\label{eq3.23}
\{P_x(\a_i) \cdot \dot x= 0,\, P_x^{(1)}(\a_i) \cdot \dot x = 0,\, \dots ,\,  P^{(j_i -2)}(\a_i) \cdot \dot x = 0\} 
\end{eqnarray}
in the variables $\dot x := (\dot x_0, \dots , \dot x_{d-1})$. This space $V(\a_i, x)$ is viewed as the subspace of $T_x\R^d$, where $(\a_i, x)$ satisfies (\ref{eq3.20}) (with $u$ being replaced with $\a_i$).

We need to verify that  the system $L(\{\a_i, j_i\}_i)$,
formed by collecting all the systems $L(\a_i, j_i)$ for individual $V(\a_i, x)$'s, is of the maximal possible rank, so that the spaces $\{V(\a_i, x)\}_i$ are in general position. 

Let $m := \sum_i (j_i - 1)$. The matrix of the system $L(\{\a_i, j_i\}_i)$ is a generalized Vandermonde  $(m \times d)$-matrix. Its rank is $m$---the maximal possible.

Let us validate this fact. Consider the auxiliary $u$-polynomial $$T(u) := P_x \cdot \dot x = \sum_{k =0}^{d-1} \dot x_k u^k.$$ By (\ref{eq3.23}), the solutions $\dot x$ of $L(\{\a_i, j_i\}_i)$ correspond exactly to the $u$-polynomials $T(u)$ with coefficients $\dot x_0, \dots  , \dot x_{d-1}$ and the roots $\{\a_i\}$ of multiplicities $\{j_i -1\}$. In other words, $T(u)$ are the $u$-polynomials that are divisible by the real polynomial $$S(u) := \prod_i (u-\a_i)^{j_i -1}$$ of degree $m$.

The equations (\ref{eq3.22}) represent the requirement that $R(u)$, the the remainder of this division, vanishes (a priori, $R(u)$ is  a polynomial of degree $m  -1$ at most). The quotient $T(u)/S(u)$ consists of polynomials of the form $Q(u, y) = \sum_{l=0}^q  y_l\, u^l$, where $q = d -1 - m$ and $\{y_l\}_l$ are some real coefficients. We can view them as free variables,  since any choice of $y_l$'s produces the polynomial $T(u, y) := S(u)\cdot Q(u, y)$, which gives rise to a solution $\dot x(y)$ of the system $L(\{\a_i, j_i\}_i)$. Moreover, different parameters $y$ produce different polynomials $S(u)\cdot Q(u, y)$. Thus $y$ gives a $(d-m)$-parametric representation of the solution space of the system $L(\{\a_i, j_i\}_i)$.

So (\ref{eq3.22})  is solvable, and the rank of its matrix equals $m$. Therefore, the spaces $\{V(\a^\star, x^\star)\}$ form a general position configuration in $\R^d$. As a result, the field $\d_u$ is generic with respect to the boundary $\d X = \{P(u, x) = 0\}$.
\end{proof}

\begin{lemma}\label{lem3.6} Let $P(u, x)$ be the polynomial in (\ref{eq3.19})  of degree $d = \sum_i j_i$. Put $m : = \sum_i (j_i - 1)$. Consider the smooth manifold $X \subset \R \times \R^m$, given by  the polynomial inequality $\{P(u, x)  \leq 0\}$. 

Then, in the vicinity of the line $L_0 := \{x = 0\} \subset \R \times \R^m$, the field $\d_u$ is traversally generic with respect to the boundary $\d X = \{P(u, x) = 0\}$.
\end{lemma}

\begin{proof} We follow the outline of the arguments from the previous lemma. 

We denote by $Z$ the hypersurface $\{P(u, x) = 0\} \subset \R\times\R^m$.

Let $x_{\{i\}} := (x_{i,0}, x_{i,1},\, \dots ,\, x_{i, j_i -2})$.  By definition, the polynomial $P(u, x)$ of degree $d = \sum_i j_i$ is the product of  the factors $$P_i(u, x_{\{i\}}) := (u - \a_i)^{j_i} + \sum_{l=0}^{j_i - 2} x_{i, l} (u - \a_i)^l.$$

Consider the hypersurfaces  $Z_i := \{P_i(u, x_{\{i\}}) = 0\}$. There exists an $\e > 0$ such that the hypersurfaces $\{Z_i\}_i$ are disjointed in the cylinder  $$\Pi_\e := \big\{\{\|x_{\{i\}}\| < \e\}_i\big\} \subset \R\times\R^d$$ with the polydisk base $\big\{B_\e := \{\|x_{\{i\}}\| < \e\}_i\big\} \subset \R^m$.  In other words, for such choice of $\e > 0$, $$Z \cap \Pi_\e = \coprod_i (Z_i\cap \Pi_\e).$$ Indeed, the claim follows from Lemma \ref{lem3.2}, since all the roots $\{\a_i\}$ of $P(u, 0)$ are distinct (note that the intersection $L_0 \cap Z = \coprod_i (\a_i, 0)$) and the roots of nearby polynomials $P(u, x)$, $x \in B_\e$, are grouped around the $\a_i$'s. 

Therefore, for any $x^\star \in B_\e$, the intersection of the line $$L_{x^\star} := \{x = x^\star\} \subset \R \times \R^m$$ with $Z$ is of the form $\coprod_i \Big(\coprod_{k \in A_i} ((\b_{ik}, x^\star)\big)\Big)$, where $\{\b_{ik}\}_k$ are the real roots  of $P_i(u, x^\star)$. 

Let us denote the multiplicity of the root $\b_{ik}^\star = \b_{ik}(x^\star)$ by $j^\star_{ik}$. 

Since, in the vicinity of the point $(\a_i, 0)$, the equations $\{P(u, x) =0\}$ and $\{P_i(u, x) = 0\}$ determine the same solution sets, we can reduce the study of locus $\Sigma_{[j^\star_{ik}]} \subset Z$ to the study of a similar locus $\Sigma_{[j^\star_{ik}]} \subset Z_i$. As in (\ref{eq3.23}), both loci are given by the equations:
$$\{P_i(u, x)=0,\, P_i^{(1)}(u, x) = 0,\, \dots ,\, P_i^{(j^\star_{ik}-1)}(u, x) = 0\}.$$

Thus we have reduced our settings to the ones studied in Lemma \ref{lem3.5}. Let us adapt equations (\ref{eq3.20})-(\ref{eq3.22}) to our present environment\footnote{Note the difference: in the present setting, the polynomial $\tilde P_i$ is depressed, i.e., the coefficient $x_{i, j_i -1}$ of $(u - \a_i)^{j_i - 1}$ is zero.}. We denote by $\tilde P_i$ the $\a_i$-shifted polynomial, so that $\tilde P_i(u-\a_i, x) := P_i(u, x)$. In particular, the vector $P_x := (u^{d-1}, \dots , u, 1)$ in (\ref{eq3.21}) and (\ref{eq3.22}) must be replaced by the vector $$\xi_i(u) := (\tilde P_i)_{x_{\{i\}}} = ((u -\a_i)^{j_i-2}, \dots , (u -\a_i), 1)$$
We denote by $\{\xi_i^{(l)}(u)\}_l$ its multiple derivatives with respect to the variable $u$. 

As in  (\ref{eq3.22}), we conclude that each vector $(\dot u, \dot x)$, tangent to the locus $\Sigma_{[j^\star_{ik}]}$ at the point $(\b_{ik}, x^\star)$ is determined by its $D\pi$-projection $\dot x \in T_{x^\star}(\R^m) \approx \R^m$. 

Therefore, $D\pi$ maps each tangent space $T_{(\b_{ik}^\star, x^\star)}(\Sigma_{[j^\star_{ik}]}^\circ)$ bijectively to the linear subspace of $T_{x^\star}\R^m$, given by the $j_{ik}^\star - 1$ homogeneous equations
\begin{eqnarray}\label{eq3.24}
V(\b_{ik}^\star) := \big\{\xi_i(\b_{ik}^\star) \cdot \dot x_{\{i\}}=0,\;\; \xi_i^{(1)}(\b_{ik}^\star)\cdot \dot x_{\{i\}} = 0,\; \nonumber \\ \dots \; ,\,  \xi_i^{(j^\star_{ik}-2)}(\b_{ik}^\star)\cdot \dot x_{\{i\}} = 0\big\}.  
\end{eqnarray}
with respect to the $j_i - 1$ variables $\dot x_{\{i\}}$. Note that if some $j_{ik}^\star = 1$, then the simple root $\b_{ik}^\star$ does not contribute to the equations.

To conclude the proof, we need to show that  these spaces $\{V(\b_{ik}^\star)\}_{i,k}$ are in general position in $\R^d$. 

Again, the argument is a modification of a similar proof from Lemma \ref{lem3.5}. Consider the entire collection of spaces $\{V(\b_{ik}^\star)\}_{i, k}$ and the collection $\mathcal L(x^\star) = \{ \mathcal L_{ik}(x^\star)\}_{i, k}$ of equations as in (\ref{eq3.24}) with respect to the $m := \sum_i (j_i -1)$ variables $\{\dot x_{\{i\}}\}_i$. The number of such equations is $m^\star := \sum_{i, k}(j_{ik}^\star - 1)$. 

So we have to show that the rank of the matrix of the system $\mathcal L(x^\star)$ is $m^\star$, the maximal possible. Note that $m$ is the reduced multiplicity of the $\d_u$-trajectory through the point $(0, 0)$, while $m^\star \leq m$ is the reduced multiplicity of the $\d_u$-trajectory through the nearby point $(0, x^\star)$. 

The equations (\ref{eq3.24}) are equivalent to the requirement that, for each $x^\star$, the $u$-polynomial $T_i(u, \dot x) := \xi_i(u) \cdot \dot x_{\{i\}}$ of degree $j_i - 2$ is divisible by the polynomial $$S_i(u) := \prod_k (u - \b_{ik}^\star)^{j^\star_{ik} - 1}$$ of degree $m_i^\star := \sum_k (j^\star_{ik} - 1)$. So the quotient $Q_i := T_i(u, \dot x)  / S_i(u)$ is a polynomial of the form $$Q_i(u, y_{\{i\}}) := \sum_{l = 0}^{j_i - 2 - m_i^\star} y_{il}\, u^l$$
, which depends on $(j_i - 1 - m_i^\star)$ parameters $\{y_{il}\}_l$. 

On the other hand, for any choice of these parameters $\{y_{il}\}_l$, we can form the polynomial $S_i(u)\cdot Q_i(u, y_{\{i\}})$ and to compute its Taylor polynomial $\sum_{l =0}^{j_i -2} a_{il}(u-\a_i)^l$ at the point $\a_i$. Then $\{\dot x_{il} = a_{il}\}_l$ must satisfy $\mathcal L_i(x^\star)$.

Thus any solution of the system $\mathcal L(x^\star)$ is generated via this linear mechanism by choosing the $m - m^\star$ parameters variables $\{y_{\{i\}}\}_i$. Different choices of $\{y_{\{i\}}\}_i$ lead to different ordered sets of $u$-polynomials $\{S_i(u)\cdot Q_i(u, y_{\{i\}})\}_i$ and therefore to different solutions $\{\dot x_{il}\}$ of $\mathcal L(x^\star)$. Therefore, the rank of the matrix of $\mathcal L(x^\star)$ is $m - m^\star$, the maximal possible. In turn, this implies that the $\{V(\b_{ik}^\star)\}_{i,k}$ are in general position in $\R^m$. 
\end{proof}

\begin{definition}\label{def3.5} We say that a field $v \in \mathcal V_{\mathsf{trav}}(X)$ is \emph{versal} if each $v$-trajectory $\g$ has a $\hat v$-adjusted neighborhood $U \subset \hat X$, equipped with special  coordinates 
$$(u,\;\underbrace{x_{10}, \dots  , x_{1j_1-2}},\, \dots \, \underbrace{x_{i0}, \dots , x_{ij_i - 2}}, \, \dots \, \underbrace{x_{p0}, \dots \, x_{pj_p - 2}},\, \underbrace{y_1, \dots, y_{n - m'(\g)}})$$
as in Lemma \ref{lem3.4} in which $\d_1X$ is given by the equation (\ref{eq3.19})\footnote{In particular, for any $\g$, we require that the reduced multiplicity $m'(\g) \leq \dim(X) - 1$.}. 
\smallskip

We denote by $\mathcal V_{\mathsf{vers}}(X) \subset \mathcal V_{\mathsf{trav}}(X)$ the subspace of all versal fields on the manifold $X$. \hfill\qed
\end{definition}

 \begin{corollary}\label{cor3.1} If a field $v$ is versal in a $v$-adjusted neighborhood of its trajectory $\g$, then there is another $v$-adjusted neighborhood $U \subset \hat X$ of $\g$ so that $v$ is traversally generic in $U$ (with respect to $\d_1X \cap U$). 
\end{corollary}

\begin{proof} By the definition of a versal field, in special coordinates with the core $\g$, the boundary $\d_1X$ is given by a polynomial equation $P = 0$ as in (\ref{eq3.19}).  By applying  Lemma \ref{lem3.6} to $P$, we conclude that $v$ is traversally generic in a smaller $v$-adjusted neighborhood $U \subset \hat X$ of $\g$. 
\end{proof}

The next proposition claims that if a field is traversally generic \emph{at} a trajectory $\g$, then it is traversally generic in its vicinity.

\begin{corollary}\label{cor3.2} Let $v$ be a traversing field on $X$ which is traversally generic \emph{at} a trajectory $\g$ (see Definition \ref{def3.4}).  

Then there exists a $\hat v$-adjusted neighborhood $U$ of $\g$, such that the field $v$ is traversally generic in $U$ with respect to $\d_1X \cap U$. 
\end{corollary}

\begin{proof} If $v$ is traversally generic at $\g$ (as the hypotheses of the corollary spell out), then by  Lemma \ref{lem3.4}, $v$ is versal in the vicinity of $\g$. By  Corollary \ref{cor3.1}, the field is traversally generic in the vicinity of $\g$.
\end{proof}

\begin{theorem}\label{th3.1} A vector field on $X$ is versal if and only if it is traversally generic; in other words, $\mathcal V_{\mathsf{vers}}(X) =  \mathcal V^\ddagger(X)$. 
\end{theorem}

\begin{proof} By Lemma \ref{lem3.4}, any traversally generic field is versal. By Corollary \ref{cor3.2}, any versal field is traversally generic in the vicinity of each trajectory $\g$, and therefore, is traversally generic globally. 
\end{proof}

In this chapter, our ultimate goal is to improve upon the Morin Theorem \ref{th3.1} by showing that, for a compact smooth manifold $X$, the space $\mathcal V_{\mathsf{vers}}(X) = \mathcal V^\ddagger(X)$ of versal/traversally generic vector fields is actually \emph{open} and \emph{dense} in the space $\mathcal V_{\mathsf{trav}}(X)$ of all traversing fields (see Theorem \ref{th3.5} below). We would like to provide the reader with an arguments that does not rely heavily the jet magic of the singularity theory in general, and on the Boardman Stratification Theory \cite{Bo} in particular. However, reluctantly, we are going to use few implications of that theory.

Let $M, N$ be smooth manifolds. For a given smooth map $\Phi: M \to N$, we denote by $\Sigma_{[j]}(\Phi)$ the locus in $M$ that is defined inductively by the formula 
$$\Sigma_{[j]}(\Phi) := \big\{x \in \Sigma_{[j -1]}(\Phi) \, \big |\; \textup{rk}\big(D_x\Phi: \Sigma_{[j -1]}(\Phi) \to N\big) \leq \dim M - j\big\}$$ 
,  where $D_x\Phi$ denotes the differential of the map $\Phi$ at $x$. Note that this definition presumes that $\Sigma_{[j -1]}(\Phi)$ is a smooth manifold. 

Bordman Theorem \cite{Bo} claims that the maps  $\Phi \in C^\infty(M, N)$, for which all the loci $\Sigma_{[j]}(\Phi)$'s are smooth manifolds, form a residual set in $C^\infty(M, N)$. When $M$ is compact, this set is open and dense in  $C^\infty(M, N)$.

Recall that ``$\Sigma_{{\underbrace{1 \dots 1}_{j}}}(\Phi)$"  is a more conventional notation for the locus $\Sigma_{[j]}(\Phi)$.

Note that in general, $\Sigma_{[j]}(\Phi)$ differs from the set $$\Sigma_{(j)}(\Phi) := \{x \in M |\; rk (D_x\Phi) \leq n - j\}$$
, the locus where the rank of the differential $D\Phi$ drops by $j$. 
Both loci, $\Sigma_{[j]}(\Phi)$ and $\Sigma_{(j)}(\Phi)$, are two basic examples of the, so called,  \emph{Boardman strata}. 
\smallskip

Given a smooth map $\Phi: M \to N$, for each $x \in M$, we denote by $j_\Phi(x)$ the maximal integer $j$ such that $x \in \Sigma_{[j]}(\Phi)$.

\begin{definition}\label{def3.6} Let $M, N$ be a  smooth $n$-dimensional manifolds and let $\Phi: M \to N$ be a smooth map. 
We say that $\Phi$ is \emph{traversally generic} if
\begin{itemize}
\item for each $1 \leq j \leq n$,  the singular set $\Sigma_{[j]}(\Phi)$ is a regularly embedded submanifold of $M$ of codimension $j$, 
\item  the map $\Phi: \Sigma_{[j]}(\Phi) \setminus \Sigma_{[j+1]}(\Phi) \to N$ is an immersion,
\item for each $y \in N$, the images of the maps 
$$\big\{\Phi : \Sigma_{[j_\Phi(x)]}(\Phi) \to N\big\}_{\{x \in \Phi^{-1}(y) |\;  j_\Phi(x) \geq 1\}}$$
are in general position in the target manifold $N$. \hfill\break
Equivalently, for each $y \in N$, the images of the tangent spaces $$\{T_x(\Sigma_{[j_\Phi(x)]}(\Phi))\}_{\{x \in \Phi^{-1}(y) |\;  j_\Phi(x) \geq 1\}}$$ under the linear monomorphisms $$\big\{D_x\Phi : T_x(\Sigma_{[j_\Phi(x)]}(\Phi)) \to T_y(N)\big\}_{\{x \in \Phi^{-1}(y) |\;  j_\Phi(x) \geq 1\}}$$ are in general position in the tangent space $T_y(N)$. 
\end{itemize}

Let us denote the space of traversally generic maps by the symbol $\mathcal G^\ddagger(M, N)$. 

\hfill\qed
\end{definition}

The next theorem is a well-known but nontrivial implication of the \emph{Boardman Maps Theory} (see \cite{Bo}, \cite{GG}).

\begin{theorem}\label{th3.2} Let $M$ be a $n$-dimensional compact manifold, and $K \subset M$ a closed set. Assume that $\Phi_0: M \to \R^n$ is a smooth map which is traversally generic in the vicinity of $K$. 

Then the traversally generic maps $\Phi: M \to \R^n$ that coincide with the given map $\Phi_0$ in the vicinity of $K$ form an open and dense set $\mathcal G^\ddagger_{\Phi_0}(M, \R^n)$ in the space $C^\infty_{\Phi_0}(M, \R^n)$ of all smooth maps  which coincide with $\Phi_0$ in the vicinity of $K$. 
\end{theorem}

\begin{proof}  When $K = \emptyset$, the theorem is a special case of Theorem 5.2 from \cite{GG}. That theorem claims that the set of all \emph{Boardman maps} (see \cite{Bo} for the relevant definitions) satisfying the \emph{normal crossing condition} (the ``NC condition" on page 157 in \cite{GG}) is residual (open and dense when $X$ is compact) in $C^\infty(X, Y)$.  By their very definition, the Boardman maps satisfy the first two bullets of Definition \ref{def3.6}, while the third bullet is exactly the formulation of the normal crossing property.

When $K \neq \emptyset$, the openness of traversally generic maps in the space $C^\infty(M, \R^n)$ implies that traversally generic maps which coincide with $\Phi_0$ in the vicinity of $K$ are open in the subspace $C^\infty_{\Phi_0}(M, \R^n) \subset C^\infty(M, \R^n)$. 

Tracing the proof of Theorem 5.2 from \cite{GG}, the density of  $\mathcal G^\ddagger_{\Phi_0}(M, \R^n)$ in $C^\infty_{\Phi_0}(M, \R^n)$ is validated by the following general observations. If, for a map from $\Phi: M \to \R^n$,  its $N$-jet $j^N(\Phi)$ is transversal to some variety $W \subset Jet^N(M, \R^n)$ (the transversality of $j^N(\Phi)$ to this variety is equivalent to the property of $f$ being a Bordman map)  at a compact set $K \subset M$, then $\Phi$ can be perturbed to a map $\Phi'$ such that $\Phi' = \Phi$ in the vicinity of $K$ and $j^N(\Phi')$ is transversal to $W \subset Jet^N(M, \R^n)$ everywhere on $M$. Using Thom's  Multijet Transversality Theorem (see \cite{GG}, Theorem 4.13), similar extension principle applies to maps that satisfy the normal crossing condition.
\end{proof}

Let $A, B$ be topological spaces. Recall that a map $\Psi: A \to B$  is called \emph{quasi-open}, if the interior $\textup{int}(\Psi(U)) \neq \emptyset$ for any open set $U \subset A$.
\smallskip

The next three simple lemmas will prepare us for the pivotal Lemma \ref{lem3.10} below. 

\begin{lemma}\label{lem3.7} Let  $A$ and $B$ be two topological spaces. Assume that there exists a pair of continuous maps, $\Psi: A \to B$ and $\Theta: B \to A$, such that $\Psi \circ \Theta = Id_B$. Then the $\Psi$ is a quasi-open map.
\end{lemma}

\begin{proof} By the continuity of $\Theta$, for any open set $U \subset A$, the set $\Theta^{-1}(U)$ is open in $B$. The identity $\Psi \circ \Theta = Id_B$ implies that $\Theta^{-1}(U) \subset \Psi(U)$. The same identity implies that if $U \neq \emptyset$ then $\Theta^{-1}(U) \neq \emptyset$.  So the $\Psi$-image of any non-empty open set contains a non-empty open set.
\end{proof}

\begin{lemma}\label{lem3.8} Let $\Psi: A \to B$ be a continuous quasi-open map. Let a subset $E \subset B$ be open and dense in $B$. Then $\Psi^{-1}(E)$ is open and dense in $A$.
\end{lemma}

\begin{proof} By the continuity of $\Psi$, the set $\Psi^{-1}(E)$ is open. Assume to the contrary that $\Psi^{-1}(E)$ is not dense in $A$. Then there exists an open set $U$ such that $U \cap \Psi^{-1}(E) = \emptyset$. Consider the set $\Psi(U)$. Since the map $\Psi$ is quasi-open, $\Psi(U)$ must contain a nonempty open subset $V$. Since $E$ is dense, the open set $V \cap E \neq \emptyset$. Therefore there exists a point $a \in U$ such that $\Psi(a) \in V \cap E$. This implies that there is $a \in \Psi^{-1}(E) \cap U$, a contradiction with the assumption that $\Psi^{-1}(E) \cap U = \emptyset$. 
\end{proof}

By combining the previous two lemmas, we get the following corollary.

\begin{lemma}\label{lem3.9} Let  $A$ and $B$ be two topological spaces. Assume that there exists a pair of continuous maps, $\Psi: A \to B$ and $\Theta: B \to A$, such that $\Psi \circ \Theta = Id_B$. Let a subset $E \subset B$ be open and dense in $B$. Then $\Psi^{-1}(E)$ is open and dense in $A$. \hfill\qed
\end{lemma}

The next lemma claims that each trajectory $\g$ of a traversing boundary generic field $v$ has a pair of special neighborhoods $V \subset U$ such that the field has an arbitrary $C^\infty$-small perturbation which is  supported in $U$ and is traversally generic with respect to $\d_1X \cap W$.  It is important to clarify the nature of the claim: \emph{first} one chooses the right neighborhoods,  and \emph{then}, in these neighborhoods, \emph{arbitrary small} perturbations of $v$ with the desired properties are available.  

As usually, we extend a given traversing field $v$ on $X$ to a germ of manifold $\hat X$ that properly contains $X$ so that the extended field $\hat v$ is traversing in $\hat X$. 

\begin{lemma}\label{lem3.10} Let  $v$ be a traversing and boundary generic field on $X$, and $\g$ its trajectory. Assume that  $v$ is traversally generic in the vicinity of a closed $\hat v$-adjusted set $F \subset \hat X$. 

Then there is a triple $W \subset V \subset U$ of compact $\hat v$-adjusted neighborhoods of $\g$ in $\hat X$ with the following properties: 
\begin{itemize}
\item there exists an arbitrary $C^\infty$-small and $U$-supported perturbation $\hat v'$ of $\hat v$, such that $v' := \hat v'|_{X}$ still is traversing and boundary generic,
\item $\hat v'$ is \emph{traversally generic} with respect to $\d_1X \cap V$, 
\item for any $\hat v'$-trajectory $\g'$ which intersects $W$,  $\g' \cap \d_1X \subset \d_1X \cap V$,
\item $\hat v' = \hat v$ in the vicinity of $F \cup (\hat X \setminus U)$ in $\hat X$. 
\end{itemize}
\end{lemma}

\begin{proof} Let $\dim(X) = n+1$. We denote by $D^n_r$ the closed Euclidean $n$-ball of radius $r$.

For any traversing vector field $v \in \mathcal V^\dagger(X)$, each trajectory $\g$ has regular $\hat v$-adjusted neighborhoods $W \subset V \subset U \subset \hat X$ such that $W \subset \textup{int}(V)$ and  $V \subset \textup{int}(U)$. With the help of the $\hat v$-flow, $U$, $V$, and $W$  are diffeomorphic to the cylinders $[0, 1] \times D^n_1$,  $[0, 1] \times D^n_{0.5}$ and $[0, 1] \times D^n_{0.25}$, respectively. 

Let $f: U \to [0, 1]$ denote the height function produced by this diffeomorphism, so that $df(\hat v) > 0$.  Let $S := f^{-1}(0)$ be a transversal section of the $\hat v$-flow in $U$.

Put $\d U := f^{-1}(\d D^n_1 \times [0, 1])$ and $\delta U := f^{-1}(D^n_1 \times \d[0, 1])$. 

Evidently, one can pick the tube $U$ with the core $\g$ so narrow that $\delta U \subset \hat X \setminus X$.

Consider two sets: 
$$\d^F U := \d U \cup (F \cap U)\;\, \text{and} \; \, \d^F S := \d S \cup (F \cap S).$$

\smallskip

Let $\mathcal V_f(U, \d^F U)$ be the space of all smooth vector fields $w$ in $U$ such that: 
\begin{itemize}
\item the field $w|_{\d^F  U}$ is proportional to the  given field $\hat v|_{\d^F  U}$\footnote{and thus tangent to $\d^F  U$.},  
\item $df(w) > 0$ in $U \setminus \delta U$, 
\item $w|_{\delta U} = 0$.
\end{itemize}
\begin{eqnarray}\label{eq3.25}
\end{eqnarray}

Such fields $w$ generate $1$-parameter flows $\Phi_w : U \times \R \to U$ for which $F \cap U$ is invariant. 
\smallskip

We may assume that $\hat v \in \mathcal V_f(U, \d^F  U)$ without modifying $v$ in $X$.  The $\hat v$-flow $\Phi_{\hat v}$ defines the obvious submersion $p_{\hat v} : U \to S$.

In fact, using the product structure in $U$, there exists a Riemannian metric $g$ on $U$ such that the gradient $\nabla_g(f) = \hat v$. (In such metric, the submersion $p_{\hat v} : U \to S$ is a harmonic map.)
\smallskip

We denote by $\textup{Sub}\big((U, \d^F U), (S, \d^F S)\big)$ the space of submersions $p: U \to S$ whose restriction to $\d^F U$  is equal to the given map $p_{\hat v}$.  Since each $w \in \mathcal V_f(U, \d^F U)$, with the help of the flow, defines a submersion $p_w$ whose restriction  to $\d^F U$ is prescribed and equals to $p_{\hat v}$, we get a continuous map 
\begin{eqnarray}\label{eq3.26}
J: \mathcal V_f(U, \d^F U) \to \textup{Sub}\big((U, \d^F U), (S, \d^F S)\big).
\end{eqnarray} 

This map $J$ is surjective. Indeed, the fibers of any submersion $p \in \textup{Sub}\big((U, \d^F U), (S, \d^F S)\big)$ form an oriented $1$-dimensional foliation $\mathcal F(p)$.  Using the product metric $g$ on $U$, we form a unit vector field $w(p)$, tangent to the fibers of $\mathcal F(p)$. By multiplying $w(p)$ by an appropriate universal smooth function $\phi: U \to \R_+$, such that $\phi > 0$ in $U \setminus \delta U$ and $\phi |_{ \delta U} = 0$, we produce a vector field $\tilde w(p) \in \mathcal V_f(U, \d^F U)$. Evidently, $p_{\tilde w(p)} = p$.

This construction $p \Rightarrow \mathcal F(p) \Rightarrow \phi\cdot w(p)$ defines a continuous map $$K: \textup{Sub}\big((U, \d^F U), (S, \d^F S)\big) \to \mathcal V_f(U, \d^F U)$$, such that $J \circ K = Id$.  Therefore, by Lemmas \ref{lem3.7}-\ref{lem3.8}, $J$ is a quasi-open map. 
\smallskip

For any $p \in\textup{Sub}\big((U, \d^F U), (S, \d^F S)\big)$,  the restriction of $p$ to $\d_1X \cap U$  produces a continuous restriction operator
\begin{eqnarray}\label{eq3.27}
\Psi: \textup{Sub}\big((U, \d^F U),\; (S, \d^F S)\big) \to C^\infty\big((\d_1X \cap U, \d_1X \cap \d^F U),\; (S, \d^F S)\big). \nonumber \\
\end{eqnarray}

Let us show that $\Psi$ is a \emph{quasi-open} map. To validate this claim, guided by Lemma \ref{lem3.9}, we will construct a continuous extension operator 
\begin{eqnarray}\label{eq3.28}
\Theta: C^\infty\big((\d_1X \cap U, \d_1X \cap \d^F U),\; (S, \d^F S)\big) \to C^\infty\big((U, \d^F U),\; (S, \d^F S)\big) \nonumber \\
\end{eqnarray}
whose composition with the restriction operator $$\Psi: C^\infty\big((U, \d^F U),\; (S, \d^F S)\big) \to C^\infty\big((\d_1X \cap U, \d_1X \cap \d^F U),\, (S, \d^F S)\big)$$ 
is the identity. The construction of operator $\Theta$ will depend on the choice of a map $p: (U, \d^F U) \to (S, \d^F S)$.
\smallskip

Consider a regular neighborhood $\mathcal N(\d_1X)$ of the submanifold $\d_1X \cap U$ in $U$. The neighborhood  fibers over its core $\d_1X \cap U$ with the fiber being a segment. Formula (\ref{eq3.7})  in Lemma \ref{lem3.1} implies that the hypersurfaces $\d U$ and $\d_1 X$ are transversal. Hence, we can choose the product structure $\pi: \mathcal N(\d_1X) \approx (\d_1 X \cap U) \times [-1, 1]$ so that the intersection $\d_1X \cap U =  \pi^{-1}\big((\d_1 X \cap U) \times \{0\}\big)$ and $\mathcal N(\d_1X) \cap \d U$ is entirely built of  fibers. 

For any smooth function $h: \d_1 X \cap U \to \R$, we aim to construct its canonical smooth extension $\hat H: \mathcal N(\d_1X) \to \R$. With this goal in mind, consider a $1$-parameter family $\{\phi_a: [-1, 1] \to \R\}_a$ of smooth bell-shaped functions such that $\phi_a(-1) = 0 = \phi_a(1)$ and $\phi_a := a \cdot \phi_0$ (so the parameter $a$ being the height $\phi_a(0) = a$ of the bell). Thus $\phi_{a +b} = \phi_a + \phi_b$. Note that, for $a > 0$, the function $\phi_a > 0$ in $(0,1)$; for $a < 0$, the function $\phi_a < 0$ in $(0,1)$; for $a= 0$, $\phi_0 = 0$.  

For a given $h: \d_1 X \cap U \to \R$, define a function $\tilde H: (\d_1 X \cap U) \times [0, 1] \to \R$ by the formula $\tilde H(x, t) := \phi_{h(x)}(t)$. Then put  $\hat H := \pi \circ \tilde H$. This function $\hat H : \mathcal N(\d_1X) \to \R$ extends to a smooth function $H: U \to \R$ by letting $H$ vanish on the complementary set $U \setminus  \mathcal N(\d_1X)$. 

Let $$\mathcal E: C^\infty(\d_1X \cap U, \R) \to C^\infty(U, \R)$$ be the continuous extension operator, defined by the formula $\mathcal E(h) := H$ for all $h \in C^\infty(\d_1X \cap U, \R)$. Since $\phi_a := a \cdot \phi_0$,  the operator $\mathcal E$ is linear.  Evidently, the operator $\mathcal E$ gives rise to a continuous linear operator $$\mathcal E_n: C^\infty(\d_1X \cap U, \R^n) \to C^\infty(U, \R^n).$$

This operator $\mathcal E_n$ will be instrumental in the construction of  extension operator $\Theta$ from (\ref{eq3.28}). In a sense, $\mathcal E_n$ can be viewed as ``the variation operator" for the operator  $\Theta$. 
\smallskip

For a given submersion $p \in \textup{Sub}\big((U, \d^F U),\, (S, \d^F S)\big)$, consider its restriction $\Psi(p) \in  C^\infty\big((\d_1X \cap U, \d_1X \cap \d^F U),\, (S, \d^F S)\big).$  

Note that, for a space $Y$, the difference $g_1 - g_0$ between two maps, $g_0: Y \to \R^n$ and $g_1: Y \to \R^n$, makes sense as a map from $Y$ to $\R^n$. 

Since the flow section $S \subset \R^n$,   for any two maps $$h, h' \in C^\infty\big((\d_1X \cap U, \d_1X \cap \d^F U),\, (S, \d^F S)\big)$$, the difference $h' - h: (\d_1X \cap U, \d_1X \cap \d^F U) \to (\R^n, 0)$ is well-defined. Take $h := \Psi(p)$ and consider the map $$p + \mathcal E_n(h' - \Psi(p)): (U, \d^F U) \to (\R^n , \d^F S).$$
Note that the restriction of this map to $\d^F U$ indeed takes $\d^F U$ to $\d^F S$:  by definition, the restrictions of $\Psi(p)$ and $h'$ to $\d_1X \cap \d^F U$ coincide, and $p$ takes  $\d_1X \cap \d^F U$ to $\d^F S$.  

Finally, we define the operator $\Theta := \Theta(p)$ in (\ref{eq3.28}) by the formula  $$\Theta(h') :=  p + \mathcal E_n(h' - \Psi(p)).$$

Due to the properties of $\{\phi_a\}_a$,  $\mathcal E_n(0) = 0$. Thus, $\Theta(\Psi(p)) = p$.  By the linearity of the operator $\mathcal E_n$ and using that $\Psi \circ \mathcal E_n = Id$, we get
 $$\Psi(\Theta(h')) := \Psi\big(p + \mathcal E_n(h' - \Psi(p))\big)  = \Psi(\mathcal E_n(h')) + \Psi(p - \mathcal E_n(\Psi(p)))$$
$$ = h' + \Psi(p) - \Psi(p) = h'.$$

Since $\Psi \circ \Theta = Id$, by Lemma \ref{lem3.7}, the restriction operator $\Psi$ is a quasi-open map. 

Consider the subset $$\mathcal G^\ddagger_V \subset C^\infty\big((\d_1X \cap U, \d_1X \cap \d^F U),\; (S, \d^F S)\big)$$ such that, for any $h \in \mathcal G^\ddagger_V$, the restriction $h: V \cap \d_1X \to S$ is a traversally generic map in the sense of Definition \ref{def3.6}. By Theorem \ref{th3.4}, $\mathcal G^\ddagger_V$ is an open and dense subset. Since $\Psi$ is a quasi-open map, by Lemmas \ref{lem3.7}-\ref{lem3.9}, $\Psi^{-1}(\mathcal G^\ddagger_V)$ is open and dense in 
the space $C^\infty\big((U, \d^F U),\; (S, \d^F S)\big)$. Since the space of submersions $\textup{Sub}\big((U, \d^F U),\, (S, \d^F S)\big)$  is open in the space $C^\infty\big((U, \d^F U),\, (S, \d^F S)\big)$, we conclude that the set $$\mathcal E^\ddagger_V := \Psi^{-1}(\mathcal G^\ddagger_V) \cap \textup{Sub}\big((U, \d^F U),\; (S, \d^F S)\big)$$ is open and dense in the space $\textup{Sub}\big((U, \d^F U),\; (S, \d^F S)\big)$.

Let us revisit  the map $J$ from (\ref{eq3.26}). Recall that $J$ is a quasi-open map. Therefore,  $J^{-1}(\mathcal E^\ddagger_V)$ is open and dense in the space $\mathcal V_f(U, \d^F U)$.

Thus we have shown that the fields that are traversally generic with respect  $\d_1X \cap V$ form an open and dense set in the space $\mathcal V_f(U, \d^F U)$. Let us spell out what 
this claim means:  for fields $\hat v' \in J^{-1}(\mathcal E^\ddagger_V)$,  each $\hat v' $-trajectory $\g$ intersects with $\d_1X \cap V$ along a finite set of points $\{a_i \in \d_{j_i}X(\hat v')^\circ\}_i$ in such a way that the differential $D p_{\hat v'}$ places the tangent spaces $\{T_{a_i}\big(\d_{j_i}X(\hat v')^\circ\big)\}_i$ in general position in $T(S)$. This behavior still leaves out an option:  some $\hat v' $-trajectory, passing through a point of $\d_1X \cap V$, may hit $\d_1X \cap U$ at a point $b \in \d_kX(\hat v')^\circ$ such that $b \notin V$. The $D p_{\hat v'}$-image of the tangent space $T_b(\d_kX(\hat v')^\circ)$ may not be in general position with respect to $D p_{\hat v'}$-images of the spaces $\{T_{a_i}\big(\d_{j_i}X(\hat v')^\circ\big)\}_i$.

To control better the flows in the vicinity of $W$, consider the fields $w \in  \mathcal V_f(U, \d^F U)$, having the following ``property $\mathsf A$": if a $w$-trajectory $\g'$ has a nonempty intersection with the set $W$ (recall that $W \subset V \subset U$), then the intersection $\g' \cap \d_1X \subset \d_1X \cap \textup{int}(V)$.  Let us denote by $\mathcal A_{V,W}$ the set of fields $w$ possessing  the property $\mathsf A$. 

Since $W$ is compact and properly contained in the interior of $V$, the set $\mathcal A_{V,W} \subset \mathcal V_f(U, \d^F U)$ is open due to the smooth dependence of the solutions of ODE's on the initial data and coefficients (non-vanishing vector fields). 

Therefore $J^{-1}(\mathcal E^\ddagger_V) \cap \mathcal A_{V,W}$ is open in $\mathcal V_f(U, \d^F U)$ and dense in $\mathcal A_{V,W}$. Evidently, the original field $\hat v \in \mathcal A_{V,W}$. So $\hat v$ admits an arbitrary $C^\infty$-small perturbation $\hat v'$ which belongs to $J^{-1}(\mathcal E^\ddagger_V) \cap \mathcal A_{V,W}$. By the definitions of all relevant spaces, $\hat v'$ possesses the properties listed in the lemma.
\end{proof}
\smallskip
 
Let $v \in \mathcal V^\dagger(X) \cap \mathcal V_{\mathsf{trav}}(X)$ and let $z: \hat X \to \R$ be as in (\ref{eq3.18}). Then along each trajectory $\g$, in the appropriate coordinates $(u, y)$, the locus $\d_1X$ can be described as the zero set of the function
\begin{eqnarray}\label{eq3.29}
z(u, y) = \prod_{a_i \in \g \cap \d_1X} \Big[ (u - u(a_i))^{j_i} + \sum_{l = 0}^{j_i - 2} \; \phi_{i, l}(y) (u - u(a_i))^l \Big]
\end{eqnarray}
, where $j_i$ denotes the multiplicity of $z|_\g$ at $a_i$, and $\phi_{i, l}(0) = 0$. 

We have seen the crucial role played in the previous arguments by the Jacobi $m'(\g) \times n$ matrix $D\Phi_\g$ whose rows are the vectors $\{\nabla_y \phi_{i, l}(0)\}_{i, l}$. Based on Lemma \ref{lem3.3}, we can give  still another interpretation to the condition $\textup{rk}(D\Phi_\g) = s$.  

As before, let $\g \cap \d_1X = \{a_i\}$, where $\{a_i\}$ are ordered by $v$. For each pair $(a_i, a_{i'})$,  in the vicinity of $\g$, consider the germ of the $\hat v$-flow generated diffeomorphism $\Psi_{i, i'}(\g): \hat X \to \hat X$ that takes $a_{i'}$ to $a_i$. 
In fact, the flow sections $S_i, S_{i'}$ at $a_i$ and $a_{i'}$ can be chosen so that $\Psi_{i, i'}(\g)(S_{i'}) = S_i$.

With the function $z: \hat X \to \R$ as in (\ref{eq3.18}) in place, at each $a_i \in \d_{j_i}X^\circ$, we  consider the 1-form $dz$ and its successive $v$-directed Lie derivatives 
$$\mathcal L_v(dz),\; \mathcal L_v^2(dz),\; \dots ,\;  \mathcal L_v^{j_i -2}(dz),$$ 
viewed as elements of the cotangent space $T^\ast_{a_i}(X)$.  Note that points $a_i \in \d_1X^\circ$ do not contribute to this list (there are two such points at most). By Theorem \ref{th3.3} below,  $\textup{rk}(D\Phi_\g)$ is the dimension of the space spanned by the $1$-forms
\begin{eqnarray}\label{eq3.30}
\Big\{\Psi_{i, 1}^\ast(\g )\big(dz |_{T_{a_i}S_i}\big),\; \Psi_{i, 1}^\ast(\g)\big(\mathcal L_v(dz)|_{T_{a_i}S_i}\big),\; \dots \nonumber \\
 \dots , \; \Psi_{i, 1}^\ast(\g)\big(\mathcal L_v^{j_i -2}(dz)|_{T_{a_i}S_i}\big)\Big\}_i
\end{eqnarray}
in the  cotangent space $T^\ast_{a_1}(S_1)$ of the section $S_1$. 
\smallskip

Theorem \ref{th3.3} expands the scope of this observation and incorporates the main claim from Theorem \ref{th3.1}. 

\begin{theorem}\label{th3.3} For a boundary generic traversing field $v \in \mathcal V^\dagger(X) \cap \mathcal V_{\mathsf{trav}}(X)$ the following properties are equivalent: 
\begin{itemize}
\item $v$ is traversally generic in the sense of Definition \ref{def3.2}, 
\item $v$ is versal in the sense of  Definition \ref{def3.5},
\item for each $v$-trajectory $\g$, $\textup{rk}(D\Phi_\g) = m'(\g)$\footnote{See formula (\ref{eq3.2}).}, where the Jacobi matrix  $D\Phi_\g$ is produced as in Lemma \ref{lem3.3}  from the representation (\ref{eq3.19}),
\item for each $v$-trajectory $\g$, the dimension of the space spanned by the  $1$-forms in (\ref{eq3.30}) is $m'(\g)$.
\end{itemize}
\end{theorem}

\begin{proof} By Theorem \ref{th3.1}, if $v$ is versal, then it is traversally generic and vice versa. So the first two bullets in the theorem are equivalent.

By the argument in Lemma \ref{lem3.4}, the property of a traversing $v \in \mathcal V^\dagger(X)$ being traversally generic implies that the boundary $\d_1X$, in the vicinity of each $v$-trajectory $\g$ and in $\hat v$-adjusted coordinates $(u, y) \in  \R \times \R^n$, is given by an equation 
$$\prod_{a_i \in \g \cap \d_1X } \big[(u - u(a_i))^{j_i} + \sum_{l = 0}^{j_i -2}\phi_{i, l}(y)(u - u(a_i))^l \big] = 0$$
such that the map $\Phi_\g: \R^n \to \R^{m'(\g)}$, produced by the functions $\{\phi_{i, l}(y)\}_{i,l}$, has the Jacobi matrix $D\Phi_\g(0)$ of the maximal rank $m'(\g)$ at the origin $0 \in \R^n$. Further reasoning in Lemma \ref{lem3.4} has established that $\textup{rk}(D\Phi_\g(0)) = m'(\g)$ implies the existence of better coordinates $(u, x, \tilde y)$ in the vicinity of $\g$, the coordinates in which the boundary $\d_1X$ is given by a simpler $\tilde y$-independent equation
$$\prod_{a_i \in \g \cap \d_1X } \big[(u - u(a_i))^{j_i} + \sum_{l = 0}^{j_i -2}x_{i, l}(u - u(a_i))^l \big] = 0.$$
This is exactly the property of $v$ being versal. On the other hand, by the argument as in Lemma \ref{lem3.6}, the versality of a vector field field implies that $\textup{rk}(D\Phi_\g(0)) = m'(\g)$ 

Thus the equivalence of the properties in the first three bullets has been proven.
\smallskip

Finally, the equivalence of the last bullet with the rest is basically implied by Lemma \ref{lem3.3}. Here a computation that establishes the last equivalence.

Let $v \in \mathcal V^\dagger(X)$. Let us choose special coordinates $(u, \vec x) \in \R \times \R^n$ in the vicinity of a typical $\g$ so that the ``versal" boundary equation is given by (\ref{eq3.19}). Consider a typical multiplier $P_i$ in (\ref{eq3.19}). Then 
$$\frac{\d P_i}{\d \vec  x} =  \sum_{l = 0}^{j_i - 2} \big(u -  u(a_i)\big)^l\; \frac{\d x_{i, l}}{\d \vec x}.$$ 
Thus 
$$d P_i  =  j_i\big(u - u(a_i)\big)^{j_i -1}du + \Big\langle \frac{\d P_i}{\d \vec x}, d\vec x \Big\rangle$$ 
and its restriction to the section $S_i := \{u = u(a_i)\}$ is given by $\big\langle \frac{\d P_i}{\d \vec x}, d\vec x \big\rangle$. 
Therefore $$\mathcal L_{\d_u}^s(d P_i)|_{S_i} = \sum_{k = 1}^n \frac{\d^s }{\d u^s} \frac{\d P_i}{\d x_k}\, dx_k.$$
As a result, $\textup{rk}(D\Phi_i(\g))$ is equal to  the dimension of the space spanned by $$\{\mathcal L_{\d_u}^s(d P_i)|_{S_i} \in T^\ast_{a_i}(S_i)\}_{0 \leq s \leq j_i -2}.$$   
 
Similarly, $\textup{rk}(D\Phi(\g))$ equals  the dimension of the space spanned by 
$$\{\Psi_{a_i, a_1}^\ast(\mathcal L_{\d_u}^s(d P_i)|_{a_i}) \in  T^\ast_{a_1}(X)\}_{i,\; 0 \leq s \leq j_i - 2}.$$

By Lemma \ref{lem3.3}, the latter space coincides  with the space spanned by $$\{\Psi_{a_i, a_1}^\ast(\mathcal L_{\d_u}^s(d z)|_{a_i}) \in  T^\ast_{a_1}(X)\}_{i,\, 0 \leq s \leq j_i - 2}$$, where $$z(u, \vec x) =  Q(u, \vec x) \times \prod_i P_i(u, \vec x))$$ with $Q(u, 0) \neq 0$.

For any field $v \in \mathcal V^\dagger(X)$, this proves that the properties described in the last two bullets are equivalent.
\end{proof}
\smallskip

It is possible to globalize the local construction, defined by formulas (\ref{eq3.8}) - (\ref{eq3.9}). 

Let $X$ be a $(n+1)$-dimensional smooth compact manifold. Let $z: \hat X \to \R$ be a smooth function as in Lemma \ref{lem3.1}. Consider the sequence of already familiar functions: 
\begin{eqnarray}\label{eq3.31}
\psi_0 := z,\, \psi_1 := \mathcal L_v(\psi_0),\, \psi_2 :=  \mathcal L_v(\psi_1),\, \dots ,\, \psi_n := \mathcal L_v(\psi_{n-1}).
\end{eqnarray}
They gives rise to the smooth maps
\begin{eqnarray}\label{eq3.32} 
\Psi(v, z) := (\psi_0, \dots , \psi_n): X \to \R^{n +1}, \nonumber \\
\Psi^\d(v, z) := (\psi_1, \dots , \psi_n): \d_1 X \to \R^n. 
\end{eqnarray}

As in (\ref{eq3.8}) - (\ref{eq3.10}), the locus $\d_1X$ is defined by the equation $\{\psi_0 = 0\}$, the locus  $\d_2X$ by the equations $\{\psi_0 = 0, \psi_1 = 0\}$, and so on. The  locus $\d_jX$ is defined by the equations $\{\psi_0 = 0, \psi_1= 0, \dots , \psi_{j - 1} = 0\}$. We notice that $\d_j^+X$ is characterized by the additional inequality $\psi_j \geq 0$. Recall that, unlike the maps in (\ref{eq3.32}), these loci do not depend on the choice of the auxiliary function $z: \hat X \to \R$.

For a $(n+1)$-dimensional $X$, the formulas (\ref{eq3.32}) gives rise to the continuous maps
$$\mathbf{\Psi}_z: \mathcal V(X) \to C^\infty(X,\, \R^{n+1}),$$ 
$$\mathbf{\Psi}^\d_z: \mathcal V(X) \to C^\infty(\d_1X,\, \R^n).$$ 

Let $\psi_0, \psi_1, \dots \psi_{l -1}$ be the standard coordinates in $\R^l$. Consider the complete flag $$\mathsf F^l= \{\R^l := F_0 \supset F_1 \supset F_2 \dots \supset F_l := \{0\}\}$$, where $F_j \subset \R^l$ is defined by the equations $$\{\psi_0 = 0,\, \dots ,\, \psi_{j -1} = 0\}.$$ Each space $F_j$ is divided by $F_{j+1}$ into two halves: $F_j^+$ and $F_j^-$; the half-space $F_j^+$ is characterized by  the inequality $\psi_j \geq 0$. 
\smallskip

Let $\mathsf{Diff}_+^{\mathsf F}(\R^l)$ denote the group of smooth diffeomorphisms of $\R^l$ that preserve all the half-spaces $\{F^\pm_j\}$ invariant. 

\begin{definition}\label{def3.7} Let $M^k$ be a smooth compact $k$-manifold.  We say that a map $\Psi: M^k \to \R^l$ is $\mathsf F$-\emph{stable} if, for an open neighborhood $\mathcal O$ of $\Psi$ in $C^\infty(M^k, \R^l)$ and each $\Psi' \in \mathcal O$, there exists a smooth diffeomorphism $\chi: M^k \to M^k$ and a diffeomorphism $\phi \in \mathsf{Diff}_+^{\mathsf F}(\R^l)$ such that 
$$\phi \circ \Psi' = \Psi \circ \chi.$$ \hfill\qed
\end{definition}

\noindent{\bf Remark 3.4.} Let $\mathbf{\Psi}_z: \mathcal V^\dagger(X) \to C^\infty(X,\, \R^{n+1})$ be the map which takes each boundary generic field $v$ to the map $\Psi(v, z)$ from (\ref{eq3.32}). 

Evidently, if the map $\Psi(v, z): \d_1X \to \R^n$ is $\mathsf F^{n+1}$-stable, then for each $v' \in \mathbf{\Psi}_z^{-1}(\mathcal O)$, the appropriate $\chi$ from Definition 3.7  will map each stratum $\d_j^\pm X(v')$ to the stratum $\d_j^\pm X(v)$. 

By definition, the $\mathsf F^{n+1}$-stable maps form an open set in $C^\infty(X, \R^{n+1})$. However, in general, they do not form a dense subset. Recall that even the common stable maps $f: Y \to Z$, where $\dim(Y) = \dim(Z) = n+1$, are dense in the space of all smooth maps $C^\infty(Y, Z)$ only for $n < 8$ (see \cite{GG}, page 163.)! 

Nevertheless, in Theorem \ref{th3.4} below, we will establish the local stability of stratifications $\{\d_j^\pm X(v)\}_j$ in the vicinity of \emph{any} $v \in \mathcal V^\dagger(X)$---a much weaker property than the $\mathsf F^{n+1}$-stability of the map $\Psi(v, z)$.  
\hfill\qed 

\smallskip

By the arguments from Lemma \ref{lem3.1}, for any $v \in \mathcal V^\dagger(X)$, the map $\Psi(v, z): X \to \R^{n+1}$ is transversal to each space $F_j$ from the flag $$\mathsf F^{n+1} := \{\R^{n+1} \supset F_1 \supset F_2 \supset  \dots \supset F_{n+1}  = \{0\}\}$$
, and $\Psi(v, z)^{-1}(F_1) = \d_1X$. We describe this transversality by saying that $\Psi$ is ``transversal to the flag $\mathsf F^{n+1}$".

\smallskip

\noindent{\bf Question 3.1.} Which maps $\Theta: X \to \R^{n +1}$, transversal to the flag $\mathsf F^{n+1}$, have the form $\Psi(v, z)$ for some function $z$ as in Lemma \ref{lem3.1} and $v \in \mathcal V^\dagger(X)$? For some $v \in \mathcal V^\dagger(X) \cap \mathcal V_{\mathsf{trav}}(X)$?
\hfill\qed

\smallskip

The answer to this question eludes us. However, if we extend the function list in (\ref{eq3.32}) by introducing an additional function $\psi_{n+1} := \mathcal L_v(\psi_n)$, we will get an extension $\hat \Psi(z, v): X \to \R^{n+2}$ of the map $\Psi(z, v): X \to \R^{n+1}$.

 According to the lemma below, for a given ``extended" map $\hat\Theta: X \to \R^{n+2}$, the field $v$ such that  $\hat\Theta = \hat\Psi(z, v)$ is often unique. 

\begin{lemma}\label{lem3.11} Let $X$ be a compact $(n+1)$-dimensional manifold.  For a smooth map $\hat\Theta: X \to \R^{n+2}$, consider the composition $\Theta := \pi \circ \hat\Theta$, where $\pi: \R^{n+2} \to \R^{n +1}$ is the projection $(\psi_0, \dots , \psi_{n +1}) \to (\psi_0, \dots , \psi_n)$. 

Assume that $\Theta$ is transversal to the flag $\mathsf F^{n+1}$ and that $\textup{rk}(D\Theta) = n + 1$ on a dense subset $A$ of $X$. Then there is at most one boundary generic field $v$ on $X$ such that $\Theta = \Psi(z, v)$ and $\hat\Theta = \hat\Psi(z, v)$. 
\end{lemma}

\begin{proof} We denote by $\theta_j$ the $j$-th component of the given map $\hat\Theta$. Let $g$ be a Riemannian metric on $X$, such that $\nabla_g(\theta_0) = \nu$, the unitary inward normal to $\d_1X$. Such metric $g$ exists since $\Theta$ is transversal to $F_1$.

If  $\Theta = \Psi(z, v)$, then $\theta_0 := z$, and $v$ must satisfy the equations: 
$$\theta_1 = \langle\nabla_g \theta_0, v\rangle,\; \theta_2 = \langle\nabla_g \theta_1, v\rangle,\; \dots \,, \theta_{n} = \langle\nabla_g \theta_{n -1}, v\rangle. $$ 
We impose an additional relation $$ \theta_{n+1} = \langle\nabla_g \theta_{n}, v\rangle$$
which reflects the assumption that $\Theta$ extends to $\hat\Theta:= \hat\Psi(z, v)$.

Since the fields $\nabla_g \theta_0, \nabla_g \theta_2, \dots \nabla_g \theta_{n}$  are assumed to be independent on the dense set $A \subset X$, the field $v$ is uniquely determined there by its scalar products with the fields $\{\nabla_g \theta_j\}_{0 \leq j \leq n}$. Since $A$ is dense, by continuity, there is at most a single field $v$ on $X$, such that $\hat\Theta = \hat\Psi(z, v)$. 
\end{proof}
\smallskip

Recall that, for a boundary generic field $v$ on $X$ and its trajectory $\g$, each point $a \in \g \cap \d_1X$ acquires some multiplicity $j(a)$. If the set $\g \cap \d_1X$ is finite and $\g$ is not a close trajectory (for instance, if $v \in \mathcal V_{\mathsf{trav}}(X) \cap \mathcal V^\dagger(X)$), then  the points of $\g \cap \d_1X$ are ordered. So we get an ordered sequence of  points $a_i \in \g \cap \d_1X$, together with their multiplicities $j(a_i)$. We call such weighted sequence $D_\g$ of points on $\g$ the \emph{divisor of} $\g$. Its degree is the multiplicity $m(\g) := \sum_{a \in \g \cap \d_1X} j(a)$ of $\g$ (see Definition \ref{def3.1}).

\begin{theorem}\label{th3.4} Let $X$ be a compact smooth $(n + 1)$-manifold with boundary. 
\begin{itemize}
\item The space $\mathcal V^\dagger(X)$ of  boundary generic fields\footnote{see Definition \ref{def2.1}.} is open and dense in the space $\mathcal V(X)$ of all vector fields on $X$. 
\item The smooth type of the Morse stratification $\{\d_jX(v)\}_j$ is locally constant within each path-connected component of  $\mathcal V^\dagger(X)$.
\item For any field $v \in \mathcal V^\dagger(X)$, there is a neighborhood $E$ of $\d_1X$ in $\hat X$, such that  the portion $\g \cap E$ of every  $\hat v$-trajectory $\g$ has the $E$-localized multiplicity $m(\g \cap E) \leq n+1$. Moreover, for each point $a \in \d_kX^\circ(v)$,  $m(\g \cap E) \leq k$ for all $\g$'s in the vicinity of the point $a$. 
\item Each point $a \in \d_kX(v)^\circ$ has a $\hat v$-adjusted neighborhood $U$ such that any  divisor $D$ in $\R$ with the properties $\deg(D) \leq k$ and $\deg(D) \equiv k \; (2)$ is realized, up to a diffeomorphism of $\R$, as the divisor $D_{\g \cap U}$ of the portion $\g \cap U$ for some $\hat v$-trajectory $\g$.
\end{itemize}
\end{theorem}

\begin{proof} We fix a Riemannian metric $g$ on $\hat X$. Consider the functions $\{\psi_i\}$ from (\ref{eq3.31}) and the map $\Psi = \Psi(v, z)$ from (\ref{eq3.32}) that they generate. As in (\ref{eq3.7})-(\ref{eq3.9}), the property of $v$ being boundary-generic is equivalent to the requirement that the gradient vectors $\nabla \psi_0, \nabla \psi_1, \dots , \nabla \psi_{j-1}$ are linearly independent on the solution set of $$\{\psi_0 = 0, \psi_1 = 0, \dots , \psi_{j-1} = 0\}$$ for all $j$. 

In terms of the complete flags $\mathsf F^{n+1}$ (with $F_0 := \R^{n+1}$), a field $v \in \mathcal V^\dagger(X)$ if and only if, for each subspace $F_j\subset \R^{n+1}$, the rank of the differential $D\Psi := D\Psi(v, g)$, being restricted to the  bundle, normal to the set $\d_jX(v) = \Psi^{-1}(F_j)$ in $X$, is equal to $j$.  

Let us denote by $\mathcal V_{\{|_\d \neq 0\}}(X)$ the space of fields that do not  vanish on the boundary $\d_1X$. Evidently, any smooth section of the tangent bundle $T(X)$ can be approximated by a smooth section that does not vanish on $\d_1X$. Therefore it suffices to show that $\mathcal V^\dagger(X)$ is a dense (in the $C^\infty$-topology) subset of $\mathcal V_{\{|_\d \neq 0\}}(X)$ in order to conclude that $\mathcal V^\dagger(X)$ is a dense subset of $\mathcal V(X)$. 

Given any field $v$ which does not vanish on $\d_1X$, we decompose it along $\d_1X$ into the normal component $\nu_1$ and the tangent component $v_1$. Then we extend the decomposition   $v = v_1 \oplus \nu_1$ in a collar of $\d_1X$ in $\hat X$. It is possible to perturb $\nu_1$ to make sure that it defines a section of the normal bundle $\nu(\d_1X, \hat X)$ that is transversal to its zero section. The perturbation can be smoothly extended into a collar of $\d_1X$ in $\hat X$, where it will be supported.  Let us use the same notations for the perturbed field.\footnote{In fact, the transversality of $\nu_1$ to $\d_1X$ defines an open set in the space $\mathcal V(\hat X)$.}  The component $\nu_1$ will be fixed in the further perturbations. Its zero locus is the manifold $\d_2X$. Next, we consider the orthogonal decomposition $v_1 = v_2 \oplus \nu_2$, where $v_2$  is tangent to $\d_2X$ and $\nu_2$ is a section of the normal bundle $\nu(\d_2X, \d_1X)$. Again, it is possible to extend this decomposition into a collar  of $\d_2X$ in $\d_1X$.  Then we perturb $\nu_2$ to make it transversal to $\d_2X$ and extend this perturbation first into a collar of $\d_2X$ in $\d_1X$ and then further into a collar of $\d_1X$ in $\hat X$. The zero locus of $\nu_2$ defines a submanifold  $\d_3X \subset \d_2X$. Continuing this sequence of perturbations, we will produce a field from $\mathcal V^\dagger(X)$. Therefore such fields form a dense set in the space $\mathcal V_{\{|_\d \neq 0\}}(X)$, and thus, in the space of all fields $\mathcal V(X)$ (this fact can be derived from the Morin Theorem \ref{th2.1}  as well). 

To show that $\mathcal V^\dagger(X)$  is open, we use the local model employed in the proof of Lemma \ref{lem3.1}. Evidently, the linear independence of the gradient fields in (\ref{eq3.10})  in the vicinity of the solution set of (\ref{eq3.9}) is an open property imposed on the pair of functions $u(z, x), w(z, x)$ from that model. By Lemma \ref{lem3.1}, locally, this independence of fields is exactly the property of $v$ to define transversal sections of the quotient $1$-bundles $\{T(\d_{j - 1}X)/T(\d_jX)\}_j$. 

By covering a compact collar $E$ of $\d_1X$ in $\hat X$ with a finite system of compact coordinate charts $(z, x)$ as in the proof of Lemma \ref{lem3.1}, we conclude that  if (\ref{eq3.9}) and (\ref{eq3.10}) are satisfied by a field $v$ in each of the charts, then all sufficiently $C^\infty$-small perturbations of $v$ will satisfy similar conditions. Thus, $\mathcal V^\dagger(X)$ is open and dense in the space $\mathcal V_{\{|_\d \neq 0\}}(X)$ which, in turn, is open and dense in the space $\mathcal V(X)$. 
\smallskip

Now let us prove the claim in the second bullet of the theorem. 

Let $\mathcal F_j^\perp$ be the family of affine spaces that are orthogonal to the subspace $F_j \subset \R^n$, and let $F_j^\perp(b) \approx \R^j$ denote its typical member, a space that contains a point $b \in F_j$. Let $\pi_j: \R^n \to \R^j$ be the orthogonal projection whose fiber is $F_j$.

Recall that, for $v \in \mathcal V^\dagger(X)$, the map $\Psi := \Psi^\d(v, z)$ in (\ref{eq3.32}) is transversal to each $F_j \subset \R^n$,  $j \in [1, n]$.  Therefore, for any boundary generic $v$, there exists an open $\e$-ball $B^j_\e := B^j_\e(v) \subset \R^j$, centered on the origin, such that the map $$\pi_j\circ \Psi:\,  (\pi_j\circ \Psi)^{-1}(B^j_\e) \to B^j_\e$$ is a surjection (i.e., its Jacobian has the maximal rank $j$).

Let $U_{j, \e} := U_{j, \e}(v)$ denote the set $(\pi_j\circ \Psi)^{-1}(B^j_\e) \subset \d_1X$. There exists a diffeomorphism $$\a_j: U_{j, \e} \approx (\pi_j\circ \Psi)^{-1}(0) \times B^j_\e$$ which identifies $U_{j, \e}$ with a space of a trivial disk bundle over the base space $$\d_jX(v) = \Psi^{-1}(F_j) := (\pi_j\circ \Psi)^{-1}(0).$$ Thus $U_{j, \e}$ is a regular neighborhood of $\d_jX(v)$ in $\d_1X$.

For any  map $\Psi: \d_1X \to \R^n$, transversal to the complete flag $\mathsf F^n$ in $\R^n$, there  exists $\e > 0$ and an open neighborhood $\mathcal O_{j, \e}(\Psi) \subset C^\infty(\d_1X, \R^n)$ of the map $\Psi$,  such that, for each map $\Psi' \in \mathcal O_{j, \e}(\Psi)$, the following properties are valid: 

\begin{itemize}
\item $\Psi^{-1}(F_j) \subset U'_{j, \e/2} \subset U_{j,\e}$, where $U'_{j, \e/2} := (\pi_j\circ \Psi)^{-1}(B^j_{\e/2})$,
\item the manifold $(\Psi')^{-1}(F_j)$ is a smooth section of the trivial $j$-disk bundle $\b_j: U_{j, \e} \to \Psi^{-1}(F_j)$\footnote{defined with the help of the trivialization $\a_j$}, a section which is transversal  to the fibers $\{\Psi^{-1}(F_j^\perp(b)) \cap U_{j, \e}\}_{b \in F_j}$ of the bundle. 
\end{itemize}
\begin{eqnarray}\label{eq3.33}
\end{eqnarray} 

The existence of the neighborhood $\mathcal O_{j, \e}(\Psi)$ routinely follows from the openness of transversal families of smooth maps on compact sets.
\bigskip

Next, we form the open set  $\mathcal O_\e(\Psi) := \cap_{j =1}^n \mathcal O_{j, \e}(\Psi)$ in the vector space $C^\infty(\d_1X,\, \R^n)$, equipped with the Whitney topology. 
\smallskip

Let $$\mathbf {\Psi}: \mathcal V^\dagger(X) \to C^\infty(\d_1X, \R^n)$$ be the continuous map that takes a field $v' \in \mathcal V^\dagger(X)$ to the map $\Psi^\d(v', z) \in C^\infty(\d_1X, \R^n)$. 

Consider the open neighborhood of $v$: $$\mathcal U(v) := \mathbf {\Psi}^{-1}\big(\mathcal O_\e(\Psi(v))\big) \subset \mathcal V^\dagger(X).$$  We claim that, for any $v' \in \mathcal U(v)$, the Morse stratifications $\{\d_jX(v)\}_j$ and $\{\d_jX(v')\}_j$ can be transformed one into another by a diffeomorphism $\Phi$ of $X$ (actually, by a diffeotopy). Let us explain how to construct the matching diffeomorphism $\Phi$. 

Put $\Psi' := \Psi^\d(v', z)$ and $\Psi := \Psi^\d(v, z)$.  For any point $a \in \d_{j+1}X(v) := (\Psi)^{-1}(F_j)$, consider the unique point $$b_j(a) := (\Psi')^{-1}(F_j) \cap \Psi^{-1}(F^\perp_{\Psi(a)}).$$ The two properties in (\ref{eq3.33}) imply that  the correspondence $\phi_j: b_j(a) \to a$ is a diffeomorphism which maps $(\Psi')^{-1}(F_j)$ to $\Psi^{-1}(F_j)$. Consider the  linear diffeotopy $\{\phi_j(a, t)\}_{t \in [0, 1]}$  that takes each point $$b_j(t, a) : = (\Psi')^{-1}(F_j)\, \cap \, \Psi^{-1}\big((1 -t)\Psi(b_j(a)) + t\Psi(a)\big)$$ to the point $a \in \d_jX(v)$.  It can be induced by an ambient isotopy  
\begin{eqnarray}\label{eq3.34}   
\{\Phi_j^t(v, v'):\; \d_1X \to \d_1X\}_{t \in [0, 1]}
\end{eqnarray}
, supported in  the neighborhood $U_{j, \e}$.  That isotopy matches $\d_{j+1}X(v)$ with $\d_{j+1}X(v')$.

Now we will improve the previous construction by building a \emph{single} ambient isotopy $\Phi^t := \Phi^t(v, v')$ that matches of all the strata $\{\d_jX(v)\}_j$ with the corresponding strata $\{\d_jX(v')\}_j$ at once. This  will be achieved in stages, indexed by $j = 2, 3, \dots $. 

Starting with the top strata $\d_2X(v')$ and $\d_2X(v)$ and letting $t = 1$ and $j =1$ in (\ref{eq3.34}), we use $\Phi_1^1: \d_1X \to \d_1X$ to match them.  Then, inside $\d_2X(v)$, we construct a diffeotopy $\Phi^t_2: \d_2X(v) \to \d_2X(v)$ which takes $\Phi^1_1(\d_3X(v'))$ to $\d_3X(v)$. 

Here is a recipe for constructing $\Phi^t_2$. Let $F_{21}^\perp(a)$ denote the line that is orthogonal to $F_2$ inside $F_1$ and contains the point $\Psi(a)$. Then $\Phi^t_2$ moves each point $\Phi^1_1(b)$, $b \in \d_3X(v')$, towards the unique point $a \in \d_3X(v)$, such that both $a$ and $\Phi^1_1(b)$ belong to the curve $S_{21}(a) :=  \Psi^{-1}(F_{21}^\perp(a))$ (the motion takes place inside  the curve $S_{21}(a)$).  Due to the choice of the neighborhoods $\mathcal O_\e(\Psi) \subset \mathcal O_{1, \e}(\Psi) \cap \mathcal O_{2, \e}(\Psi)$ (see the properties in (\ref{eq3.33})) and using that the 2-disk $F_2^\perp(a) \subset \R^n$ is sliced into segments $\{F_1^\perp(c)\}_c$, where $c \in F_{21}^\perp(a)$, we conclude that the slice $S_{21}(a)$ indeed contains a unique point $\Phi^1_1(b)$, $b \in \d_3X(v')$. 

It is possible to extend the isotopy $\Phi^t_2: \d_2X(v) \to \d_2X(v)$ to an isotopy of the ambient $\d_1X$ and even of $X$. Abusing notations, we use the same symbol for the extended isotopy. Now, the composition of $\Phi_2^1 \circ \Phi_1^1$ matches the pair $\d_2X(v') \supset  \d_3X(v')$ with the pair $\d_2X(v) \supset  \d_3X(v)$.

Following this scheme, we eventually construct a diffeomorphism which matches the two stratifications $\{\d_jX(v')\}_j$ and $\{\d_jX(v)\}_j$ for each $v' \in \mathcal U(v) := \mathbf {\Psi}^{-1}\big(\mathcal O_\e(\Psi(v))\big)$. Therefore, the smooth type of the stratification $\{\d_jX(v)\}_j$ is locally stable as the function of  $v \in \mathcal V^\dagger(X)$. Note that we do not claim that $\Psi := \Psi^\d(v, z)$ is a $\mathsf F^n$-stable map in the sense of Definition \ref{def3.7}, a much stronger assertion! 

In fact, for any path-connected component of $\mathcal V^\dagger(X)$, the smooth topological type of the Morse stratification remains constant. Indeed, if the points $v_0$ and $v_1$ are connected by a continuous path $\g: [0,1] \to \mathcal V^\dagger(X)$, then each point $\g(t)$ produces an open neighborhood $\mathcal U(\g(t)) \subset \mathcal V^\dagger(X)$ as above. Using the compactness of image $\g([0, 1])$, we can cover it by a finite number of open sets $\{\mathcal U(\g(t_i))\}_i$. By the previous arguments, any pair of Morse stratifications $\{\d_jX(v)\}_j$ and $\{\d_jX(v')\}_j$, where $v, v' \in \mathcal U(\g(t_i))$, can be transformed one into the other by a diffeomorphism of $X$. Therefore $\{\d_jX(v_0)\}_j$ and $\{\d_jX(v_1)\}_j$ can be transformed one into the other by a finite composition of locally available diffeomorphisms.
\smallskip 

Now, let us validate the last two bullets of the theorem. Formula (\ref{eq3.7}) implies that  $m(\g \cap E) \leq n+1$ for all trajectories $\g$ in $E$. Thus $m'(\g \cap E) \leq n$. By the same token, if  $a \cap\d_kX(v)^\circ$, then $m(\g \cap U_a) \leq k$ for all $\g$ in a $\hat v$-adjusted tubular neighborhood $U_a$ of $a$. Of course, this implies that  $m'(\g \cap U_a) \leq k -1$. 

The last bullet of the theorem follows from Lemma \ref{lem3.1}, in particular from the existence of special coordinates $(u, x)$ in which formula (\ref{eq3.7}) is valid.
\end{proof}

Finally, we have reached the summit of this paper.

\begin{theorem}\label{th3.5} Let $X$ be a smooth compact $(n+1)$-dimensional manifold with boundary.
\begin{itemize}
\item The subspace  $\mathcal V^\ddagger(X)$ of traversally generic fields is open and dense in the space $\mathcal V_{\mathsf{trav}}(X)$ of all traversing fields\footnote{By definition, traversing fields do not vanish on $X$.}. 
\item If $v \in \mathcal V^\ddagger(X)$, then for every $v$-trajectory $\g$, we get  $m'(\g) \leq n$ and $m(\g) \leq 2(n + 1)$. 
\end{itemize}
\end{theorem}

\begin{proof}
First we would like to show that  the space $\mathcal V^\ddagger(X)$ of traversally generic fields is open in the space $\mathcal V^\dagger(X) \cap \mathcal V_{\mathsf{trav}}(X)$ of boundary generic traversing fields, and thus by Theorem \ref{th3.4}, in the space of all traversing fields.

Let us start with a traversally generic field $v \in \mathcal V^\ddagger(X)$. By Theorem  \ref{th3.4}, the first bullet, there exists an open neighborhood $\mathcal O^\dagger(v) \subset \mathcal V(X)$ of $v$ such that $\mathcal O^\dagger(v) \subset \mathcal V^\dagger(X)$.  
\smallskip

We pick a finite set $\{S_\a\}$ of transversal sections of the $\hat v$-flow in the vicinity of $X \subset \hat X$, each section $S_\a$ being diffeomorphic to an open $n$-disk. We denote by $T_\a$ a closed $n$-disk which is properly contained in $S_\a$. We pick  the sections $\{S_\a \supset T_\a\}_\a$ so that each $\hat v$-trajectory that intersects with $X$ hits at least one of the flow sections $T_\a$ in its interior. Let $U_\a$ be the union of $\hat v$-trajectories through $S_\a$, and let $V_\a$ be the union of $\hat v$-trajectories through $T_\a$ (so that $V_a \subset U_\a$). Thus, $\{V_\a \cap X\}_\a$ is a cover of $X$. 

We denote by $p_\a(\hat v): U_\a \to S_\a$ the $\hat v$-directed projection, defined by  the formula $x \to \hat\g_x \cap S_\a$, where $\hat \g_x$ is the $\hat v$-trajectory through $x$.

Since $\hat v$ is transversal to the closure of each $S_\a$, there is an open neighborhood $\mathcal O^\star(v) \subset \mathcal V^\dagger(X)$ of $v$ such that, for every field $v' \in \mathcal O^\star(v)$,  each $\hat v'$-trajectory hits every section  $S_\a$ transversally or misses it.  Moreover, by the $C^\infty$-continuous dependence of ODE's solutions on the initial values and on the non-vanishing vector field, we can assume that each  $\hat v'$-trajectory through $T_\a$ is contained in the set $U_\a$, and each $\hat v'$-trajectory through the hypersurface $\delta V_\a := \d_1X \cap V_\a$ hits $S_\a$ transversally at a singleton. 

We form $\mathcal O^{\dagger\star}(v) := \mathcal O^\star(v) \cap \mathcal O^\dagger(v)$, an open neighborhood of $v$ in $\mathcal V^\dagger(X)$.

Consider the $\hat v$-directed maps $\{p_\a(\hat v): \delta V_\a \to S_\a\}_\a$. Since $v \in \mathcal V^\ddagger(X)$, each map $p_\a(\hat v)$ is traversally generic in the sense of Definition \ref{def3.6}  (where $M = \delta V_\a$ and $N = S_\a$). Examining Definition \ref{def3.6} and  Definition \ref{def3.2}, we see that the converse is true as well: if all $\{p_\a(\hat v): \delta V_\a \to S_\a\}_a$ are traversally generic maps, then $v$ is a traversally generic field.

By Theorem  \ref{th3.2}, there is an open neighborhood $\mathcal U_\a$ of the map $p_\a(\hat v)$ in $C^\infty(\delta V_\a, S_\a)$ such that each map $\Phi \in \mathcal U_\a$ is traversally generic. 

Consider the map 
$$\Xi: \mathcal O^{\dagger\star}(v) \to \prod_\a C^\infty(\delta V_\a, S_\a)$$
that takes each field $v' \in \mathcal O^{\dagger\star}(v)$ to the collection of maps $\{p_\a(\hat v')\}_\a$, defined by the $\hat v'$-flow. By the continuity of $\Xi$, the set $\mathcal O^\ddagger(v) := \Xi^{-1}(\prod_\a \mathcal U_\a)$ is open in $\mathcal O^{\dagger\star}(v)$ and thus in $\mathcal V^\dagger(X)$. 

Note that $\{\textup{Int}(\delta V_\a)\}_\a$ form an open cover of $\d_1X$, so that each $\hat v'$-trajectory through $X$ hits one set $\textup{Int}(\delta V_\a)$ at least. Since the property of a vector field $v'$ being traversally generic can be faithfully expressed in ``semi-local" terms of the vicinities of its $\hat v'$-trajectories,  we conclude that any $v' \in  \mathcal O^\ddagger(v)$ is traversally generic, in other words, that $\mathcal O^\ddagger(v)$, open in $\mathcal V^\dagger(X)$, is also open in $\mathcal V^\ddagger(X)$.
\smallskip

Now we would like to show that  $\mathcal V^\ddagger(X)$ is dense in the space $\mathcal V^\dagger(X) \cap \mathcal V_{\mathsf{trav}}(X)$. So we start with a boundary  generic and traversing field $v$. By Lemma 4.1 \cite{K1}, for such a field, there exists a smooth function $f: X \to \R$ so that $df(v) > 0$ in $X$. Let us denote by $\mathcal C_f$ the open neighborhood of $v$ in $\mathcal V^\dagger(X) \cap \mathcal V_{\mathsf{trav}}(X)$ defined by the inequality $\{df(v') > 0|\; v' \in \mathcal V^\dagger(X)\}$.
\bigskip

Each $\hat v$-trajectory $\g$ has a nested triple $W \subset V\subset U$ of $\hat v$-adjusted neighborhoods in $\hat X$ with the properties described in the key Lemma \ref{lem3.10}. Since $X$ is compact, we can choose a finite collection $\{W_i \subset V_i \subset U_i\}_i$ of such triples so that $\{\textup{int}(W_i)\}_i$ form a finite cover of $\d_1X$. Let us order the triples. 

We denote by $S_i$ a transversal section of the $\hat v$-flow in $U_i$. Let   by $T_i := S_i \cap V_i$, and $Q_i := S_i \cap W_i$. 

Recall some old notations: for a given $X$-traversing field $\hat w$ in $\hat X$ and a subset $A \subset \hat X$, we denote by $\hat X(\hat w, A)$ the union of $\hat w$-trajectories that pass through $A$. Let  $X(w, A) := \hat X(\hat w, A) \cap X$.

As we perturb the given boundary generic and traversing field $\hat v$, we will insist on all the perturbations $\hat v'$ being so small that:
\begin{enumerate}
\item $df(v') > 0$ (that is, $v' \in \mathcal C_f$),
\item $v' \in \mathcal V^\dagger(X)$,
\item $\hat v'$ is transversal to all the sections $S_i$,
\item the $\hat v'$-adjusted sets $\{\hat X(v', Q_i)\}_i$ cover $X$,
\item $\hat X(v', Q_i) \cap \d_1X \subset \hat X(v, T_i) \cup \d_1X$ for all $i$,
\end{enumerate}
\begin{eqnarray}\label{eq3.35}
\end{eqnarray}

We denote by $\hat{\mathcal U}^\bullet$ the set of such fields $\hat v'$ on $\hat X$. It depends on the choice of sections $\{S_i \supset T_i \supset Q_i\}_i$ and, via these sections, on the original field $\hat v$. Evidently, $\hat v \in \hat{\mathcal U}^\bullet$. In fact, $\mathcal U^\bullet$, formed by the restrictions to $X$ of the fields from $\hat{\mathcal U}^\bullet$, is an open set in the space $\mathcal V^\dagger(X) \cap \mathcal V_{\mathsf{trav}}(X)$. Indeed,  the openness of sets of fields satisfying  $(1)$ and $(3)$ is obvious, satisfying  $(2)$ follows from Theorem \ref{th3.4}, and $(4)$ and $(5)$ follows from the smooth dependence  of solutions of ODE's on initial data and on non-vanishing vector fields (on the ``coefficients").
\smallskip
 
Let us pick an arbitrary open neighborhood $\mathcal W_{\hat v} \subset \mathcal V(\hat X)$ of $\hat v$.  Put  $$\mathcal W^\bullet_{\hat v} := \mathcal W_{\hat v} \cap \hat{\mathcal U}^\bullet\; \text{and} \; \mathcal W^\bullet_{v} := \mathcal W_{\hat v} \cap \mathcal U^\bullet.$$ We intend to find a field $\hat v' \in  \mathcal W^\bullet_{\hat v}$ that is traversally generic with respect to $\d_1X$. This will prove that traversally generic fields form a dense set in $\mathcal V^\dagger(X) \cap \mathcal V_{\mathsf{trav}}(X)$.

Let us order the $\hat v$-sections  by forming a  finite list: $(S_1, S_2, \dots , S_N)$. By an inductive argument in the number $i$ of sections from the list,  we will systematically ``enlarge"  a set of trajectories which are traversally generic with respect to a growing portion of $\d_1X$.  

Here is how the induction step $i-1 \Rightarrow i$ works. Assume that we have managed to find a field $\hat v' \in  \mathcal W^\bullet_{\hat v}$ such that it is traversally generic, when restricted to the closed $\hat v'$-adjusted set $$F_{i -1}(\hat v'):= \hat X(\hat v', \coprod_{1 \leq k < i} Q_k).$$ By key Lemma \ref{lem3.10}, there is a $\hat X(\hat v', S_i)$-supported perturbation $\hat v'' \in \mathcal W^\bullet_{\hat v}$ of $\hat v'$ such that: $(1)$ $\hat v''$ is traversally generic in $\hat X(\hat v'', Q_i)$ and $(2)$ $\hat v'' = \hat v'$, when restricted to $[\hat X \setminus \hat X(\hat v',  S_i)] \cup F_{i -1}(\hat v')$. Such a field $\hat v''$ is traversally generic in $F_i(\hat v'')$.

Since the property $(4)$ from (\ref{eq3.35}) is enforced through the induction arguments, eventually (for $i = N$), we will construct a field $\hat w \in \mathcal W^\bullet_{\hat v}$ which is traversally generic everywhere in $X$. 

Note that the case $i =1$, the base of induction, is exactly the claim of  Lemma  \ref{lem3.10}.
\bigskip

 It remains to prove the last bullet of the theorem. For any field $v \in \mathcal V^\ddagger(X)$, we have shown (see (\ref{eq3.4})) that $m'(\g) \leq n$ for all trajectories $\g$. 

There are only two points of odd multiplicity in the set  $\g \cap \d_1X$---the two ends of $\g$.  Thus $m(\g)$ can be written in the form $$(2t_0 + 1) + \sum_{i=1}^q 2s_i + (2t_1 + 1).$$  As a result,  $$m'(\g) = 2t_0 + \sum_{i=1}^q (2s_i - 1)+ 2t_1 \leq n.$$ The latter inequality implies  that $q \leq n$. Therefore, $$m(\g) = m'(\g) + q + 2  \leq n+ q + 2 \leq 2n + 2$$, twice the dimension of $X$. 
\end{proof}

\begin{corollary}\label{cor3.3} For a given smooth nonsingular  function $f: X \to \R$, the traversally generic $f$-gradient-like fields form an open and dense set in the space of all $f$-gradient-like fields. 
\end{corollary}

\begin{proof} For a fixed nonsingular $f: X \to \R$, the set of $f$-gradient-like fields is defined by the inequality $df(v) > 0$, and therefore is open in the space of all fields $\mathcal V(X)$. With this remark in mind, the corollary follows from Theorem \ref{th3.5}.
\end{proof}
\bigskip

The  semi-local models of traversally generic flows that we have developed in this paper will form a foundation of our future investigations of the rich and universal combinatorics that governs  such flows. These models will also enable us to study the topology of the trajectory spaces, generated by the traversally generic flows, an interesting class of $CW$-complexes that behave as surrogate manifolds. Finally, Theorems \ref{th3.4} and \ref{th3.5} insure that the traversally generic flows are typical among the traversing flows, thus justifying these future endeavors.

\end{document}